\documentclass[a4paper,11pt]{amsart}

\usepackage{ucs}
\usepackage[utf8x]{inputenc}

\usepackage{amsmath}
\usepackage{amscd}
\usepackage{amsxtra}
\usepackage{amssymb}
\usepackage{amsthm}
\usepackage{stmaryrd}
\usepackage{amsbsy}

\newtheorem{theorem}{Theorem}[section]
\newtheorem{corollary}[theorem]{Corollary}
\newtheorem{lemma}[theorem]{Lemma}
\newtheorem{proposition}[theorem]{Proposition}
\theoremstyle{definition}

\theoremstyle{remark}
\newtheorem{example}[theorem]{Example}
\newtheorem{remark}[theorem]{Remark}
\newtheorem{question}[theorem]{Question}
\DeclareSymbolFont{largeletters}{OT1}{cmr}{m}{n}
 
\DeclareMathSymbol{\lcolon}{\mathopen}{operators}{`:}
\DeclareMathSymbol{\rcolon}{\mathclose}{operators}{`:}
\DeclareMathSymbol{\symm}{\mathop}{largeletters}{`S}
\DeclareMathSymbol{\inner}{\mathbin}{AMSa}{"79}
\DeclareMathOperator{\tr}{tr}
\DeclareMathSymbol{\symm}{\mathop}{largeletters}{`S}
\DeclareMathSymbol{\inner}{\mathbin}{AMSa}{"79}
\DeclareMathSymbol{\juxta}{\mathbin}{operators}{`\#}
\DeclareMathSymbol{\astrut}{\mathalpha}{letters}{"5F}
\DeclareMathSymbol{\astar}{\mathalpha}{symbols}{"03}
\DeclareMathSymbol{\orint}{\mathalpha}{symbols}{"22}
\DeclareMathSymbol{\orcir}{\mathalpha}{AMSa}{"09}

\newcommand{\sah}{a}
\newcommand{\sahi}{a^{-1}}

\newcommand{\ch}{\operatorname{ch}}

\newcommand{\dual}{\spcheck}

\newcommand{\End}{\mathrm{End}}

\newcommand{\Set}[1]{\left\{#1\right\}}

\newcommand{\id}{\mathrm{id}}

\newcommand{\SG}[1]{\mathfrak{S}_{#1}}
\newcommand{\sheaf}[1]{\mathcal #1}

\newcommand{\Hom}{\mathrm{Hom}}
\newcommand{\td}{\operatorname{td}}
\newcommand{\Ext}{\mathrm{Ext}}
\newcommand{\HH}{\mathrm{HH}}
\newcommand{\HT}{\mathrm{HT}}
\newcommand{\HO}{\mathrm{H\Omega}}
\newcommand{\DC}{\mathcal{D}}
\newcommand{\DCb}{\DC^{\mathrm \flat}}
\newcommand{\HKR}{\mathrm{HKR}}
\newcommand{\Ko}{\mathrm K}
\newcommand{\Alb}{\mathrm{Alb}}
\newcommand{\LieG}[1]{\mathrm{#1}}
\newcommand{\Sp}{\LieG{Sp}}

\newcommand{\Sym}{\mathrm S}
\newcommand{\HILB}{\mathrm{Hilb}}
\newcommand{\CC}{\mathbb{C}}
\newcommand{\qqed}{\hspace*{\fill}$\Box$}

\newcommand{\rk}{{\rm rk}}

\title{Remarks on derived equivalences of Ricci-flat manifolds}
\author{Daniel Huybrechts and Marc Nieper-Wi\ss kirchen}
\date{\today}

\begin{document}

    \begin{abstract}
        After a finite \'etale cover, any Ricci-flat K\"ahler manifold decomposes into
        a product of complex tori, irreducible holomorphic symplectic manifolds, and Calabi--Yau
        manifolds. We present results indicating that this decomposition is an invariant
        of the derived category. The main idea to distinguish the
        derived category of an irreducible holomorphic symplectic
        manifold from that of a Calabi--Yau manifold is that point
        sheaves do not deform in certain (non-commutative)
        deformations of the former, whereas they do for the
        latter. On the way, we prove a conjecture of C\u{a}ld\u{a}raru
        on the module structure of the
        Hochschild--Kostant--Rosenberg isomorphism for manifolds
        with trivial canonical bundle as a direct consequence of
        recent work by  Calaque,  van den Bergh, and
        Ramadoss.
    \end{abstract}

    \maketitle

    A \emph{Calabi--Yau manifold} $Z$ is a simply-connected compact K\"ahler manifold with trivial canonical
    bundle and $H^0(Z, \Omega^i_Z) = 0$ for all $0 < i < \dim Z$. Note that any Calabi--Yau manifold of dimension at least
    three is projective.

    A \emph{holomorphic symplectic manifold $Y$} is a compact K\"ahler manifold
    which admits an everywhere non-degenerate holomorphic two-form. The canonical
    bundle of such a
    $Y$ is again trivial. We call $Y$ \emph{irreducible}
    if it is in addition simply-connected and the holomorphic
    two-form is unique up to scaling. A projective holomorphic symplectic manifold will simply
    be called \emph{projective symplectic}.

    We recall the following classical decomposition theorem for Ricci-flat mani\-folds (see
    \cite{beauville:Progr,beauville:kummer}).

    \begin{theorem}
        \label{thm:decomp}
        Let $X$ be a compact K\"ahler manifold with $0 = c_1(X) \in H^2(X, {\mathbb R})$.
        Then there exists a finite \'etale cover $\widehat X \to X$ such that $\widehat X$ itself
        decomposes as
        \[
            \widehat X \cong \prod_i A_i \times \prod_j Y_j \times \prod_k Z_k,
        \]
        where the $A_i$ are simple complex tori, the $Y_j$ are irreducible holomorphic symplectic, and the $Z_k$
        are Calabi--Yau varieties of dimension at least three.
    \end{theorem}

    We would like to answer the following:
    \begin{question}
        \label{qu:general}
        Suppose $X$ and $X'$ are two smooth complex projective varieties with $c_1(X) = 0$ and $c_1(X') = 0$
        such that there exists an exact equivalence
        \[
            \DCb(X) \cong \DCb(X')
        \]
        between their bounded derived categories of coherent sheaves. Are then the symplectic factors and the
        Calabi--Yau factors of $X$ and $X'$ also derived equivalent, i.e.~
        \[
            \DCb(Y_j) \cong \DCb(Y'_{\sigma(j)})\quad\text{and}\quad\DCb(Z_k) \cong \DCb(Z'_{\tau(k)})
        \]
        for suitable permutations $\sigma$ and $\tau$?
    \end{question}
As only for projective manifolds the (derived) category of
coherent sheaves encodes enough relevant geometric information, we
restrict ourselves to this case. Clearly, in the theorem the
manifold $X$ is projective if and only if all the factors $A_i$,
$Y_j$, and $Z_k$ are projective.

To be more specific, one could first consider the following case:
Suppose $Y, Y'$ and $Z, Z'$ are irreducible projective symplectic
manifolds and Calabi--Yau manifolds, respectively. Suppose further
that there exists an exact equivalence $\DCb(Y \times Z) \cong
\DCb(Y' \times Z')$. Does this imply $\DCb(Y) \cong \DCb(Y')$ and
    $\DCb(Z) \cong \DCb(Z')$?

\begin{remark}
Allowing \'etale covers makes things more subtle. E.g.\ for $A$ an
abelian surface, one knows that the Hilbert scheme $\HILB^3(A)$ of
three points on $A$ is derived equivalent to the Hilbert scheme
$\HILB^3(\hat A)$, where $\hat A$ denotes the dual abelian surface
(see \cite{ploog:thesis}).

After an \'etale cover $\HILB^3(A)$ is isomorphic to $A \times K^2
(A)$ and similarly for $\hat A$. Here, $K^2 (A)$ denotes the
$4$-dimensional generalised Kummer variety of $A$. However, it is
unknown whether $K^2 (A)$ and $K^2(\hat A)$ are derived equivalent
(see \cite{namikawa}).
\end{remark}

So far, we are only able to show that derived equivalence does
distinguish between the three types in the decomposition theorem.
More precisely, we prove the following:
\begin{theorem}
Let $X$ and $X'$ be smooth projective varieties with equivalent
derived categories, i.e.\ there exists an exact equivalence
$\DCb(X) \cong \DCb(X')$. If $X$ is either an abelian variety or
an irreducible projective symplectic manifold, then $X'$ is of the
same type.
\end{theorem}

The result is proved by using particular algebraic properties of
Hochschild (co)homology which are invariant under derived
equivalence. Following general principles, abelian varieties and
symplectic varieties will be studied by means of Hochschild
cohomology of degree one (see Proposition \ref{mainprop:AV})
respectively degree two (see Theorem \ref{mainthm:HK}).

\medskip

On the way, we observe that a conjecture of C\u{a}ld\u{a}raru
\cite{caldararu:mukai2} can in the case of mani\-folds with
trivial canonical bundle rather easily be proved  by combining the
recent articles of van den Bergh and Calaque \cite{vdbergh} and
Ramadoss \cite{ramadoss}. C\u{a}ld\u{a}raru's conjecture asserts
 that Hochschild homology and Dolbeault cohomology are
isomorphic as modules over Hochschild cohomology via the modified
Hochschild--Kostant--Rosenberg isomorphism
$I^K:\HH^*(X)\stackrel\sim\to \HT^*(X)$, which is proved to be a
ring isomorphism  in \cite{vdbergh} (confirming a claim of
Kontsevich in \cite{kontsevich}). Putting things together one
obtains (see Corollary \ref{cor:Calconj}):

\begin{theorem}
Suppose $X$ is a manifold with trivial canonical bundle. Then the
modified Hochschild--Kostant--Rosenberg isomorphism defines an
isomorphism
$$(I^\Ko,I_\Ko):(\HH^*(X),\HH_*(X))\stackrel\sim\to
(\HT^*(X),\HO_*(X))$$ compatible with the ring structure of
$\HH^*(X)$ and $\HT^*(X)$ and the module structure of $\HH_*(X)$
and $\HO_*(X)$.
\end{theorem}

Working more locally, the proof should also lead to a proof of
C\u{a}ld\u{a}raru's conjecture for manifolds with non-trivial
canonical bundle, but this shall be pursued elsewhere.
\medskip

{\bf Acknowledgments.} We wish to thank E.\ Macr\`i and P.\
Stellari for useful discussions. Financial support of the
following institutions is gratefully acknowledged: Imperial
College London, Max-Planck-Institute Bonn, and SFB/TR 45.

\section{Hochschild (co-)homology}

This section contains only well-known material presented in a form
ready for later use. The general references are
\cite{caldararu1,vdbergh,kontsevich}. In Section \ref{Sect14} we
discuss an open conjecture of C{\u a}ld{\u a}raru on the module
structure of Hochschild homology under the
Hochschild--Kostant--Rosenberg isomorphism and prove special cases
of it which are crucial for the proof of our main result.

\subsection{Hochschild (co-)homology}\label{Sect11}
Let $X$ be a smooth projective variety of dimension $n$. By
$\iota\colon X \stackrel\sim\to \Delta \subset X \times X$, we
denote the diagonal embedding. One defines the \emph{Hochschild
cohomology of $X$} by \[\HH^*(X) := \bigoplus_i
\HH^i(X)[-i]\quad\text{with}\quad
        \HH^i(X): = \Ext^i_{X\times X}(\iota_* \sheaf O_X, \iota_* \sheaf O_X)\]
and the \emph{Hochschild homology} by
\[\HH_*(X) := \bigoplus_i \HH_i(X)[i]\quad\text{with}\quad
\HH_i(X) := \Ext^{-i}_{X\times X}(\iota_* \omega_X\dual[- n], \iota_* \sheaf
O_X).
    \]

\begin{remark}
There are other conventions on the grading of Hochschild homology
in the literature. The grading used in this paper is analogous to
the homological and the cohomological  grading in algebraic
topology. Moreover, the Mukai pairing becomes homogeneous of
degree zero in this setting.
\end{remark}

The Yoneda product, i.e.\ composition in the derived category,
endows $\HH^*(X)$ with the structure of a graded $\CC$-algebra and
$\HH_*(X)$ becomes a module over it. The product in Hochschild
cohomology shall be denoted by
$$\cup\colon\HH^p(X)\otimes\HH^q(X)\to\HH^{p+q}(X)$$
and the module structure by
$$\cap\colon\HH^p(X)\otimes\HH_q(X)\to \HH_{q-p}(X).$$

\subsection{Polyvector fields}\label{Sect12} A similar structure can be defined by using
polyvector fields and forms.  Let $\Omega_{X, *}$ denote the
complex with trivial
    differentials and $\Omega^i_X$ in degree $-i$, i.e.\
    \[
        \Omega_{X, *} := \bigoplus_q \Omega^q_X[q],
    \]
    and let $\bigwedge^* \sheaf T_X$ be its dual:
    \[\bigwedge\nolimits^* \sheaf T_X := \bigoplus_q \bigwedge\nolimits^q \sheaf T_X[-q]\cong\Omega_{X,*}\dual.
    \]
The cup-product defines a ring structure on the graded
    $\CC$-algebra
    \[
        \HT^*(X) :=\bigoplus \HT^i(X)[-i]~~~~~~{\rm with}~~~~~~\HT^i(X):=\bigoplus_{p+q=i} H^p(X, \bigwedge\nolimits^q \sheaf
        T_X)
    \]
    and contraction with polyvector fields endows
    \[
        \HO_*(X) := \bigoplus \HO_i(X)[i]~~~~~~{\rm with}~~~~~~\HO_i(X) = \bigoplus_{q-p=i}
        H^{p}(X, \Omega_X^q)
    \]
    with the structure of a $\HT^*(X)$-module. We shall write
these two operations as
$$\wedge\colon \HT^p(X)\otimes\HT^q(X)\to\HT^{p+q}(X)$$
and $$\inner\colon\HT^p(X)\otimes\HO_q(X)\to\HO_{q-p}(X).$$

 \subsection{The Hochschild--Kostant--Rosenberg
 isomorphism}\label{Sect13}

    There is a unique natural isomorphism
    \begin{equation}\label{eqn:HKRcompl}
        I\colon \iota^* \iota_* \sheaf O_X \to \Omega_{X,*}
    \end{equation}
    in $\DCb(X)$
    such that its adjoint morphism $\iota_* \sheaf O_X \to \iota_* \Omega_{X, *}$ coincides with the exponential
    of the universal Atiyah class $\alpha\colon \iota_* \sheaf O_X \to \iota_* \Omega_X[1]$ of $X$
    (see~\cite{caldararu:mukai2} or~\cite{buchweitz-flenner}).

Consider the  adjunctions $\iota^*\dashv \iota_*$ and
$\iota_!\dashv \iota^*$, where $\iota_!({\mathcal
F})=\iota_*({\mathcal F}\otimes\omega_X\dual[-n])$. Using those,
the isomorphism $I$ yields the two
\emph{Hochschild--Kostant--Rosenberg isomorphisms}
    \[
        \begin{aligned}
            I^\HKR\colon \HH^i(X) = \Ext^i_{X\times X}(\iota_* \sheaf O_X, \iota_* \sheaf O_X)
            & \stackrel{\sim}{\to} \Ext^i_X(\iota^* \iota_* \sheaf O_X, \sheaf O_X)
            \\
            & \stackrel{\sim}{\to}
            \Ext^i_X(\Omega_{X, *}, \sheaf O_X) = \HT^i(X)
        \end{aligned}
    \]
and

\[
        \begin{aligned}
          I_\HKR\colon \HH_i(X) & = \Ext^{-i}_{X\times X}(\iota_* \omega_X\dual[-n], \iota_* \sheaf O_X)
          \stackrel{\sim}{\to} \Ext^{-i}_X(\sheaf O_X, \iota^*\iota_*\sheaf O_X)
        \\
          & \stackrel{\sim}{\to}
          \Ext^{-i}_X(\sheaf O_X, \Omega_{X, *}) = \bigoplus H^{-i+j}(X,\Omega_X^j)=\HO_i(X).
        \end{aligned}
    \]

In general, these two morphisms respect neither the algebra nor
the module structure. In order to compare $(\cup,\cap)$ with
$(\wedge,\inner)$, the original HKR-isomorphisms have to be
modified as follows.
    Let $a := \hat A^{\frac{1}{2}} \in H\Omega_0(X)$, where $\hat A=\hat A(X)$ is the characteristic
    class on $X$ corresponding to the $\hat A$-genus. Then consider the induced graded isomorphisms
    \[
        a^{-1} \inner \cdot\colon \HT^*(X) \to \HT^*(X), v \mapsto a^{-1} \inner v
    \]
    and
    \[
        a \wedge \cdot\colon \HO_*(X) \to \HO_*(X), \alpha \mapsto a \wedge \alpha.
    \]
    Here, $\inner$ denotes the contraction of a polyvector field by a form
    \[
      \inner\colon  \HO_{-p}(X) \otimes \HT^q(X) \to \HT^{p + q}(X),\ (\alpha, v) \mapsto \alpha \inner v
    \]
    and $\wedge$ denotes the map
    \[
        \wedge\colon \HO_p(X) \otimes \HO_q(X) \to \HO_{p + q}(X),\ (\alpha, \beta) \mapsto \alpha \wedge \beta
    \]
induced by the exterior product of forms. Earlier we used the
contraction of  a form $\alpha$ by a polyvector field $v$, which
was denoted by $v\inner\alpha$. The order of the arguments should
avoid any confusion. We have to stress, however, that we adopt the
following sign convention for the inner product: For an odd vector
space $V$ and vectors $v\in V$, $w\in V\dual$ one has $v\inner
w=-w\inner v$, that is $\inner$ shall be symmetric in the graded sense. \label{page:conven}

Following Kontsevich (\cite{kontsevich}), one then defines the
following modified HKR-isomorphisms
 \[
        I^\Ko := (a^{-1}\inner \cdot) \circ I^\HKR\colon \HH^*(X) \stackrel{\sim}{\to} \HT^*(X)
    \]
and
    \[
        I_\Ko := (a \wedge \cdot) \circ I_\HKR\colon \HH_*(X)
        \stackrel{\sim}{\to}
        \HO_*(X).
    \]

By \cite{kontsevich} and \cite{vdbergh} one knows that $I^\Ko
\colon \HH^*(X) \stackrel{\sim}{\to} \HT^*(X)$ is an isomorphism
of $\CC$-algebras, i.e.\ $I^\Ko(v_1\cup v_2)=I^\Ko(v_1)\wedge
I^\Ko(v_2)$ for all $v_1,v_2\in \HT^*(X)$.

In Section \ref{Sect14} we will comment on the multiplicativity of
$I_\Ko$.

\begin{remark}
Recall that $\hat A(X)$ is obtained from the power series
$x/(e^{x/2}-e^{-x/2})$, whereas one uses $x/(1-e^{-x})$ for the
Todd-genus $\td(X)$. Hence $\td(X)=\exp({\rm c}_1(X)/2)\wedge\hat
A(X)$. Since inner product with ${\rm c}_1(X)$ acts as a
derivation on the algebra $\HT^*(X)$, the modified HKR-isomorphism
$I^\Ko$ is multiplicative  if and only if $(\exp(t\cdot {\rm
c}_1(X))\wedge a)\inner\cdot)\circ I^\HKR$ is multiplicative,
where $t\in{\mathbb C}$. Thus, one could as well use
$a=\sqrt{\td(X)}$ for the definition of $I^\Ko$.

Similar arguments apply to the definition of $I_\Ko$ and its
conjectured multiplicativity (see Section \ref{Sect14}).
\end{remark}

\subsection{Hochschild (co-)homology and Chern characters under derived equivalence.}
\label{Sect15} Suppose
    \[
        \Phi\colon \DCb(X_1) \stackrel\sim\longrightarrow \DCb(X_2)
    \]
is an exact equivalence of two smooth projective varieties $X_1$
and $X_2$. Due to results of C\u{a}ld\u{a}raru \cite{caldararu1}
and Orlov \cite{orlov} (cf.\ \cite[Ch.\ 6]{huybrechts:fm}), one
knows that $\Phi$ induces an isomorphism of $\CC$-algebras
    \[
        \Phi^{\HH^*}\colon \HH^*(X_1) \stackrel\sim\longrightarrow \HH^*(X_2)
    \]
    and an isomorphism of $\HH^*(X_1)$-modules (via $\Phi^{\HH^*}$)
    \[
        \Phi^{\HH_*}\colon \HH_*(X_1) \stackrel\sim\longrightarrow \HH_*(X_2).
    \]

Under the same assumptions, the following diagram commutes (see
\cite{caldararu1}):

\begin{equation}
        \label{eq:diach}
        \begin{CD}
            \DCb(X_1) @>\Phi>> \DCb(X_2) \\
            @V{\ch}VV @VV{\ch}V \\
            H\Omega_0(X_1) &&  H\Omega_0(X_2) \\
            @V{I_\HKR^{-1}}VV @VV{I_\HKR^{-1}}V \\
            \HH_0(X_1) @>>{\Phi^{\HH_*}}> \HH_0(X_2).
        \end{CD}
\end{equation}

In fact, the composition $I_\HKR^{-1}\circ\ch$, i.e.\ the
Hochschild--Chern character, admits a different description in
terms of Fourier--Mukai transforms, but we shall not need this
(see \cite{caldararu1}).


\section{The multiplicative structure of Hochschild homology}\label{Sect14}

In \cite{caldararu:mukai2} C{\u a}ld{\u a}raru conjectured that
the modified HKR-isomorphism $I_\Ko$ is an isomorphism of
$\HH^*(X)$-modules, where the $\HT^*(X)$-module $\HO_*(X)$ is
viewed as an $\HH^*(X)$-module via $I^\Ko$. In the proof of the
main result (Theorem \ref{mainthm:HK}) a special case of this
conjecture is used. So, as the general conjecture is still open,
we will have to provide an alternative argument for this point and
in fact we will give two.

Firstly, we will show how the recent paper of Ramadoss
\cite{ramadoss} can be used to establish C{\u a}ld{\u a}raru's
conjecture for varieties with trivial canonical bundle. Secondly,
as \cite{ramadoss} is technically involved and for the general
benefit of having an alternative proof, we shall also explain a
more direct approach showing multiplicativity of $I_\Ko$ for
classes $v\in \HH^2(X)$ and algebraic classes $\beta\in \HH_0(X)$
(under additional assumptions on the manifold). Fortunately, this
special case is sufficient for the proof of Theorem
\ref{mainthm:HK}. Moreover, the argument does apply also to other
manifolds whose canonical bundle is not necessarily trivial. This
justifies further the inclusion of the second argument, which
relies on graph homology and the wheeling theorem.

\subsection{$I_\Ko$ for varieties with trivial canonical bundle.}
Let us first prove that \cite{ramadoss} does imply C{\u a}ld{\u
a}raru's conjecture for varieties with trivial canoni\-cal bundles
or, more general, for certain submodules  $\HH_*'\subset\HH_*$
respectively $\HO'_*\subset\HO_*$. Let
$$\HO_*'(X)\subset \HO_*(X)$$ be the $\HT^*(X)$-submodule
generated by the subspace $\bigoplus_i H^i(X,\Omega_X^n)$.
Simi\-larly, we define
$$\HH'_*(X)\subset \HH_*(X)$$ to be the $\HH^*(X)$-submodule
generated by the subspace
$$I_\HKR^{-1}\left(\bigoplus_i H^i(X,\Omega_X^n)\right)=I_\Ko^{-1}\left(\bigoplus_i
H^i(X,\Omega_X^n)\right).$$

\begin{proposition}\label{prop:module}
The map $I_K$ defines an isomorphism of $\HH^*(X)$-modules
$$I_K:\HH_*'(X)\stackrel{\sim}{\to} \HO_*'(X),$$ where the $\HT^*(X)$-module $\HO_*'(X)$ is viewed as
an $\HH^*(X)$-module via $I^K$.
\end{proposition}

\begin{proof} We have to show that for any $\alpha\in \HH_q'(X)$ the following
diagram commutes

\begin{equation}\label{eqn:mult}\begin{CD}
\HH^p(X)@>{\mbox{}\cap \alpha}>>\HH'_{q-p}(X)\\
@V{I^\Ko}VV@VV{I_\Ko}V\\
\HT^p(X)@>>{\inner I_\Ko(\alpha)}>\HO_{q-p}'(X).
\end{CD}
\end{equation}

Suppose we have shown this already for all $\alpha_0$ with
$I_\Ko(\alpha_0)\in H^i(X,\Omega_X^n)$. Then the general result
follows from the following computation for a class of the form
$\alpha=v_0\cap\alpha_0$ with $I_\Ko(\alpha_0)\in
H^i(X,\Omega_X^n)$ and any $v\in\HH^p(X)$:
\begin{eqnarray*} I_\Ko(v\cap
\alpha)&=&I_\Ko(v\cap(v_0\cap\alpha_0))=I_\Ko((v\cup
v_0)\cap\alpha_0)\\
&\stackrel{(\sharp)} =&I^\Ko(v\cup v_0)\inner I_\Ko(\alpha_0)
\stackrel{(*)}{=}(I^\Ko(v)\wedge
I^\Ko(v_0))\inner I_\Ko(\alpha_0)\\
&=&I^\Ko(v)\inner (I^\Ko(v_0)\inner
I_\Ko(\alpha_0))\stackrel{(\sharp)}=I^\Ko(v)\inner
I_\Ko(v_0\cap\alpha_0)\\
&=&I^\Ko(v)\inner I_\Ko(\alpha).
\end{eqnarray*}
Here ($*$) is due to the multiplicativity of $I^\Ko$ (see
\cite{kontsevich} and \cite{vdbergh}) and ($\sharp$) follows from
the assumed commutativity of (\ref{eqn:mult}) for $\alpha_0$.

Thus, it remains to prove the commutativity of (\ref{eqn:mult})
for all $\alpha\in \HH_*(X)$ with $I_\Ko(\alpha)\in
H^i(X,\Omega_X^n)$.

But before, let us revisit the definition of the HKR-isomorphism
for Hochschild homology. By construction, $I_\HKR$ is given by
composing the adjunction morphism $\Ext_{X\times
X}^{-i}(\iota_*\omega_X\dual[-n],\iota_*\sheaf
O_X)=\Ext^{-i}_{X\times X}(\iota_!\sheaf O_X,\iota_*\sheaf
O_X)\cong\Ext^{-i}_X(\sheaf O_X,\iota^*\iota_*\sheaf O_X)$ in the
second variable with $I\colon\iota^*\iota_*\sheaf
O_X\stackrel{\sim}\to\Omega_{X,*}$. A priori, one could also apply
the adjunction $\iota^*\dashv\iota_*$ and composite with $I$ in
the first factor (analogously to the definition of the
HKR-isomorphism for Hochschild cohomology). More precisely, one
could compose the adjunction morphism $\Ext_{X\times
X}^{-i}(\iota_*\omega_X\dual[-n],\iota_*\sheaf
O_X)\cong\Ext_X^{-i}(\iota^*\iota_*\sheaf
O_X\otimes\omega_X\dual[-n],\sheaf O_X)$ with $I$ to obtain $\widetilde
I_\HKR\colon\HH_i(X)\stackrel\sim\to H^{-i}(X,\bigwedge^*\sheaf
T_X\otimes\omega_X[n])\stackrel\sim\to\HO_i(X)$, where we
implicitly use the natural isomorphism $R\colon \bigwedge^p\sheaf
T_X\otimes\omega_X\cong\Omega_X^{n-p}$. 

The two isomorphisms
$I_\HKR,\widetilde I_\HKR\colon\HH_i(X)\stackrel\sim\to\HO_i(X)$ differ
by the Todd genus. This is the main result of \cite{ramadoss}.
More precisely, following Ramadoss one has $${\rm td}(X)\wedge
I_\HKR=\widetilde I_\HKR.$$ Note that there is an extra sign
operator in \cite{ramadoss} which is caused by a different sign convention
for the isomorphism $R$ and thus $\widetilde I_\HKR$.

In particular
$${\rm pr}_{n}\circ I_\HKR={\rm pr}_{n}\circ\widetilde
I_\HKR,$$ where $\mathrm{pr}_{n}\colon\HO_i(X)=\bigoplus_{q-p=i}H^p(X,\Omega_X^q)\to
H^{n-i}(X,\Omega^n_X)$ denotes the projection.

Back to the commutativity of (\ref{eqn:mult}) for classes in
$H^i(X,\Omega_X^n)$.  First observe that
$H^i(X,\Omega_X^n)=\Ext^i_X(\omega_X\dual,\sheaf O_X)=
\Ext^{i-n}_X(\omega_X\dual[-n],\sheaf O_X)$ embeds via the
diagonal into $\HH_{n-i}(X)$, i.e.\ there is a natural map
$$\iota_*\colon H^i(X,\Omega_X^n)\hookrightarrow \HH_{n-i}(X).$$
Moreover, it is easy to check that the composition of $\widetilde
I_\HKR\circ\iota_*\colon H^i(X,\Omega_X^n)\to \HO_{n-i}(X)$ with the
projection $\HO_{n-i}(X)\to H^i(X,\Omega_X^n)$ is the identity.
Hence, $\alpha=\iota_*(I_\Ko(\alpha))$ for $I_\Ko(\alpha)\in
H^i(X,\Omega_X^n)$ and $I_\Ko(\iota_*\beta)=\beta$ for any
$\beta\in H^i(X,\Omega_X^n)$.

The HKR-isomorphism $I\colon\iota^*\iota_*\sheaf
O_\Delta\stackrel\sim\to\Omega_{X,*}$ yields the isomorphisms
$$\begin{array}{clcl} I(\sheaf F)\colon &\Ext^*_{X\times X}(\iota_*
\sheaf F, \iota_* \sheaf O_X)&&\\ &\cong \Ext^*_X((\iota^* \iota_*
\sheaf O_X) \otimes \sheaf F, \sheaf O_X)& \stackrel\sim\to&
\Ext^*_X(\Omega_{X, *} \otimes \sheaf F, \sheaf O_X)
\end{array}$$
        which are functorial in $\sheaf F$ (as object on $X$).
For $\sheaf F = \sheaf O_X$ this is the isomorphism $I^\HKR$ and
for $\sheaf F = \sheaf \omega_X\dual[-n]$ it is $\widetilde
I_\HKR$. (We continue to use the natural isomorphism $R\colon
\Ext^*_X(\Omega_{X, *} \otimes \sheaf \omega_X\dual[-n], \sheaf
O_X)\cong H^*(X, \bigwedge^* \sheaf T_X \otimes \omega_X[n])
\stackrel\sim\to \HO_{-*}(X)$ as in the definition of $\widetilde
I_\HKR$ above.)

By functoriality of $I(\sheaf F)$ applied to a morphism
$\beta\colon\sheaf\omega_X\dual[-n]\to\sheaf O_X[i-n]$ induced by a
class $\beta\in H^i(X,\Omega_X^n)$, we obtain
$$\widetilde I_\HKR(v\cap\iota_*\beta)=\widetilde I_\HKR(v\circ\iota_*\beta)=I^\HKR(v)\inner\beta$$
for all $v\in \HH^*(X)$ (our definition of $R$ is chosen exactly in a way such that this formula holds
without any additional signs).

Write a class $\beta$ of maximal
holomorphic degree as $\beta=I_\HKR(\alpha)=\widetilde
I_\HKR(\alpha)$. Then apply
 Lemma~\ref{lem:innermax} below to $w=I^\HKR(v)\in H^*(X,\bigwedge^j\sheaf T_X)$ to obtain:
 \[
 \begin{aligned}
 I_K(v \cap \iota_* \beta) & =
                a \wedge I_\HKR(v \cap \iota_* \beta)\\&
                 = \exp(-{\rm c}_1(X)/2)\wedge a^{-1} \wedge
                {\rm td}(X)\wedge I_\HKR(v\cap\iota_*\beta)\\
                &=\exp(-{\rm c}_1(X)/2)\wedge\left(a^{-1}\wedge\widetilde I_\HKR(v\cap
                \iota_*\beta)\right)\\
                 &=\exp(-{\rm c}_1(X)/2)\wedge\left(a^{-1}\wedge(I^\HKR(v) \inner
                 \beta)\right)
                \\
                & = \exp(-{\rm c}_1(X)/2)\wedge\left((a^{-1} \inner I^\HKR(v)) \inner \beta\right)\\
                &= \exp(-{\rm c}_1(X)/2)\wedge\left(I^\Ko(v) \inner \beta\right)\\
                &=\exp(-{\rm c}_1(X)/2)
                \wedge\left(I^\Ko(v)\inner I_\Ko(\iota_*\beta)\right).
            \end{aligned}
        \]

This proves the assertion up to the additional factor $\exp(-{\rm
c}_1(X)/2)$. To conclude, it suffices to show that
\begin{equation}\label{eqn:enough}
{\rm c}_1(X)\wedge(w\inner \gamma)=0
\end{equation}
whenever $\gamma$ has maximal holomorphic degree, i.e.\ $\gamma\in
H^*(X,\Omega_X^n)$. Using the compactness of $X$, (\ref{eqn:enough})
is equivalent to
\begin{equation}\label{eqn:integral}
\int_X\left({\rm c}_1(X)\wedge(w\inner
\gamma)\right)\wedge\delta=0
\end{equation}
for all $\delta\in \HO_*(X)$. Applying Lemma \ref{lem:innermax}
again, one finds
\begin{eqnarray*}
\left({\rm c}_1(X)\wedge(w\inner \gamma)\right)\wedge\delta=\pm
({\rm
c}_1(X)\wedge\delta)\wedge (w\inner\gamma)\\
=\pm(({\rm c}_1(X)\wedge\delta)\inner w)\inner\gamma\\
=\pm({\rm c}_1(X)\inner(\delta\inner w))\inner\gamma.
\end{eqnarray*}
Due to the assumption that $\gamma\in H^*(X,\Omega_X^n)$, only those
$\delta$ with $\delta\inner w\in H^*(X,\sheaf T_X)$ contribute to
(\ref{eqn:integral}). Then apply Lemma \ref{lem:second} below to
conclude.
\end{proof}

\begin{lemma}
        \label{lem:innermax}
        Let $\beta \in H^*(X,\omega_X)\subset\HO_n(X)$ be a form of maximal holomorphic
        degree. Then
        \[
           (-1)^\ell (\beta' \inner w) \inner \beta = \beta' \wedge (w \inner \beta)
        \]
        for all $w \in \HT^*(X)$ and all $\beta' \in H^*(X,\Omega_X^\ell)$.
\end{lemma}

\begin{proof}
The lemma clearly follows from the following easy fact in linear
algebra: Let $V$ be an $n$-dimensional vector space and let
$\beta\in\bigwedge^nV\dual$, $\beta' \in \bigwedge^\ell V\dual$,
and $w \in \bigwedge^* V$. Then
\begin{equation}
            \label{eq:linalg}
            (-1)^\ell(\beta' \inner w) \inner \beta = \beta' \wedge (w \inner
            \beta).
\end{equation}
As the equation is linear in $\beta$, $\beta'$, and $w$, we may
assume that $\beta' = e^{k-\ell+1} \wedge \ldots \wedge e^{k}$,
$\beta = e^1 \wedge \ldots \wedge e^n$, and $w = e_{k} \wedge
\ldots \wedge e_2\wedge e_{1}$. Here $(e_j)$ is a basis of $V$ and
$(e^j)$ is the dual basis. The verification of (\ref{eq:linalg})
for this choice is straightforward and left to the reader, but
remember the sign convention for the inner product on page
\pageref{page:conven}.
\end{proof}

\begin{lemma}\label{lem:second}
For any complex manifold $X$ the contraction with the first Chern
class of $X$
$${\rm c}_1(X)\inner\cdot\colon H^*(X,\sheaf T_X)\to H^{*+1}(X,\sheaf O_X)$$
is trivial.
\end{lemma}

\begin{proof}
The assertion can be proved by using the curvature as the
Dolbeault representative of the Atiyah class $A(\sheaf T_X)\in
H^1(X,\sheaf E nd(\sheaf T_X)\otimes\Omega_X)$ or, more
algebraically, by using the naturality of the Atiyah class as
follows: Any morphism $\varphi\colon \sheaf E_1\to \sheaf E_2$ of complexes yields a
commutative diagram

\begin{equation}\begin{CD}
\sheaf E_1@>{\mbox{}A(E_1)}>>\sheaf E_1\otimes\Omega_X[1]\\
@V{\varphi}VV@VV{\varphi\otimes \id}V\\
\sheaf E_2@>>{A(E_2)}>\sheaf E_2\otimes\Omega_X[1].
\end{CD}
\end{equation}

Applied to a class $v\in H^i(X,\sheaf T_X)$ viewed as a morphism
$\varphi_v\colon\sheaf O_X[-i]\to\sheaf T_X$ it shows that the
composition

$$\sheaf O_X[-i]\stackrel{\varphi_v}\longrightarrow\sheaf
T_X\stackrel{A(\sheaf T_X)}\longrightarrow\sheaf
T_X\otimes\Omega_X[1]$$ is trivial, for $A(\sheaf O_X)=0$.

As an element in $H^{i+1}(X,\sheaf T_X\otimes\Omega_X)$ this
trivial composition can also be computed as the contraction of
$A(\sheaf T_X)\in H^1(X,\sheaf E nd(\sheaf
T_X)\otimes\Omega_X)=H^1(X,(\sheaf T_X\otimes
\Omega_X)\otimes\Omega_X)$ by $v\in H^i(X,\sheaf T_X)$ in the
second factor of the product $(\sheaf
T_X\otimes\Omega_X)\otimes\Omega_X$. However, the Atiyah class
$A(\sheaf T_X)$ is symmetric in $\Omega_X$, i.e.\ $A(\sheaf
T_X)\in H^1(X,\sheaf T_X\otimes S^2(\Omega_X))$ (see e.g.\
\cite{kapranov}), and hence also the contraction of $A(\sheaf
T_X)$ by $v$ in the third factor of $\sheaf
T_X\otimes\Omega_X\otimes\Omega_X$ is trivial. Taking the trace
yields
$$0={\rm tr}(A(\sheaf T_X)\circ\varphi_v)={\rm tr}(v\inner A(\sheaf
T_X))=v\inner {\rm c}_1(X),$$ which is equivalent to the
assertion.
\end{proof}

\begin{remark}
That the contraction of $H^i(X,\sheaf T_X)$ with ${\rm c}_1(X)$
vanishes for $i=0,1$ is obvious, as it simply says that ${\rm
c}_1(X)$ stays of type $(1,1)$ under infinitesimal automorphisms
and deformations. The general case is geometrically less obvious.
Note that the contraction by ${\rm c}_1(X)$ of elements in
$H^*(X,\bigwedge^i\sheaf T_X)$ for $i>1$ will in general be
non-trivial.
\end{remark}

 Although in general different, there are interesting cases
when $\HO_*'(X)=\HO_*(X)$ and $\HH_*'(X)=\HH_*(X)$. E.g.,\ if the
canonical bundle $\omega_X=\Omega_X^n$ is trivial, then
$H^0(X,\Omega_X^n)$ gene\-rates the $\HT^*(X)$-module $\HO_*(X)$
and the image of
$$\iota_*:H^0(X,\Omega_X^n)=\Hom(\omega_X\dual,\sheaf
O_X)\hookrightarrow
\Ext^{-n}(\iota_*\omega_X\dual[-n],\iota_*\sheaf O_X)=\HH_n(X)$$
generates the $\HH^*(X)$-module $\HO_*(X)$. In this case, the
proposition yields

\begin{corollary}\label{cor:Calconj} Suppose $\omega_X\cong\sheaf O_X$.
Then the isomorphism $$I_K:\HH_*(X)\stackrel{\sim}{\to} \HO_*(X)$$
is an isomorphism of $\HH^*(X)$-modules, where the
$\HT^*(X)$-module $\HO_*(X)$ is viewed as an $\HH^*(X)$-module via
$I^K$.\qqed
\end{corollary}

\subsection{Multiplicativity of $I_\Ko$ for algebraic classes.}

As before, we denote by
    \[
        \alpha = \alpha_X\colon \iota_* \sheaf O_X \to \iota_* \Omega_X[1]
    \]
the universal Atiyah class of $X$. Composition and projection onto
the symmetric algebra of $\Omega[1]$ yield morphisms
$\alpha^n\colon \iota_* \sheaf O_X \to\iota_* \Omega_X^n[n]$ and
thus the exponential of the Atiyah class
\[
        \exp(\alpha)\colon \iota_* \sheaf O_X \to \iota_*\Omega_{X,
        *}.
    \]
Seen as a morphism between Fourier--Mukai kernels, $\exp(\alpha)$
defines a natural transformation from the identity functor to the
functor given by tensoring with $\Omega_{X, *}$. Applied to
$\sheaf F$ it yields $$\exp(\alpha)(\sheaf F)=q_* (p^* \sheaf F
\otimes \exp(\alpha)):\sheaf F\to\sheaf F\otimes\Omega_{X,*},$$
whose trace is the Chern character of $\sheaf F$
    \[
        \ch(\sheaf F)=\tr_{\sheaf F} [q_* (p^* \sheaf F \otimes \exp(\alpha))]
         \colon \sheaf O_X \to \Omega_{X, *}.
    \]
    Here $\tr_{\sheaf F}$ denotes taking the trace with respect to $\sheaf F$.

    \begin{proposition}\label{prop:secondGraph}
        \label{prop:hh2_comp} Suppose that either
        $H^0(X,\bigwedge^2\sheaf T_X)=0$ or that $X$ is
        irreducible holomorphic symplectic.
        Let $\beta \in \HH_0(X)$ be an algebraic class (that is a class that is mapped to $\ch(\sheaf F)$ under $I_\HKR\colon \HH_*(X)
        \to \HO_*(X)$ for some $\sheaf F \in \DCb(X)$). Then
        \[
            I_\Ko(u \cap \beta) = I^\Ko(u) \inner I_\Ko(\beta)
        \]
        for all $u \in \HH^2(X)$.
    \end{proposition}

 \begin{proof}
(The following proof has been inspired by the proof of~\cite[Thm.\
4.5]{caldararu:mukai2}.)

Set
        \[
            v := I^\HKR(u) \in \HT^2(X) = H^2(X, \sheaf O_X) \oplus H^1(X, \sheaf T_X)
            \oplus H^0(X, \bigwedge\nolimits^2 \sheaf T_X).
        \]
        As before we let $I\colon \iota^*\iota_*\sheaf O_X \to \Omega_{X, *}$
        denote the quasi-isomorphism
        (\ref{eqn:HKRcompl}). Then
        $v \circ I \in \Ext^*_X(\iota^* \iota_* \sheaf O_X, \sheaf O_X)$ is the morphism corresponding to
        $u$ under the adjunction $\iota^* \dashv \iota_*$, i.e.
        \[
            \iota_*(v \circ I) \circ \eta = u,
        \]
        where $\eta\colon \iota_* \sheaf O_X \to \iota_* \iota^*( \iota_* \sheaf O_X)$ is the
        natural adjunction morphism.

        With these notations, the class $\exp(\alpha)(\sheaf F)$ is given by
        \[
            I \circ q_* (p^* \sheaf F \otimes \eta) \in \Ext^0_X(\sheaf F, \sheaf F \otimes \Omega_{X, *}).
        \]
        Let us denote by $(u \circ \beta)' \in \Ext^2_X(\sheaf O_X, \iota^* \iota_* \sheaf O_X)$
        the element that corresponds to $u \cap \beta = u \circ \beta \in \HH_{-2}(X)=\Ext^2_{X\times X}(i_! \sheaf O_X, i_* \sheaf O_X)$
        under the adjunction
        $\iota_! \dashv \iota^*$. (Recall, $\iota_!(\sheaf F)=\iota_*(\sheaf F \otimes
        \omega_X\dual[-n])$ or, alternatively, $\iota_!=S_{X\times X}^{-1}\circ\iota_*\circ S_X$, where $S_Y$
        denotes the Serre functor on $Y$.)

        Let $\mu' \in \Ext^*_X(\iota^* \iota_* \sheaf O_X, \omega_X[n])$ be arbitrary and
        \[
            \mu = \iota_* \mu' \circ \eta \in \Ext^*_{X\times X}(\iota_* \sheaf O_X, \iota_* \omega_X[n])
        \]
        the corresponding element under the adjunction $\iota^* \dashv \iota_*$.

        By definition of $\iota_!$ via Serre duality on $X$ and $X\times X$, one has
        \[
            \int_X\mu' \circ (u \circ \beta)' = \int_{X \times X} \mu \circ u \circ \beta.
        \]

        It is $\beta = I_\HKR^{-1}(\ch(\sheaf F))$ the Hochschild--Chern character of $\sheaf F$
         (see~\cite{caldararu:mukai2}).
        By definition,
        \[
            \int_{X \times X} \mu \circ u \circ \beta
            = \int_X \tr_{\sheaf F} [q_*(p^* \sheaf F \otimes (\mu \circ u))].
        \]

        One has
        \[
            \begin{split}
                \int_X \mu' \circ (u \circ \beta)' & = \int_{X \times X} \mu \circ u \circ \beta \\
                & = \int_X \tr_{\sheaf F} [q_*(p^* \sheaf F \otimes (\mu \circ u))] \\
                & = \int_X \tr_{\sheaf F}
                    [q_*(p^* \sheaf F \otimes (\iota_* \mu' \circ \eta \circ \iota_* v \circ \iota_* I \circ \eta))] \\
                & = \int_X \mu' \circ \tr_{\sheaf F} [q_* (p^* \sheaf F \otimes (\eta \circ \iota_* v \circ \iota_* I \circ \eta))].
            \end{split}
        \]
        As $\mu'$ was arbitrary, it follows that
        \[
            (u \circ \beta)' = \tr_{\sheaf F} [q_* (p^* \sheaf F \otimes (\eta \circ \iota_* v \circ \iota_* I \circ \eta))].
        \]
        Applying $I$ to both sides yields:
        \begin{equation}
            \label{eq:cal}
            \begin{split}
                I_\HKR(u \cap \beta) & = I \circ (u \circ \beta)' \\
                & = \tr_{\sheaf F} [q_* (p^* \sheaf F \otimes (\iota_* I \circ \eta \circ \iota_* v \circ \iota_* I \circ \eta))] \\
                & = \tr_{\sheaf F} [\exp(\alpha)(\sheaf F) \circ v \circ \exp(\alpha)(\sheaf F)].
            \end{split}
        \end{equation}

        The rest of the proof is done by a case-by-case analysis:

\noindent {\bf The first case $v \in H^2(X, \sheaf O_X)$:} One has
        \[
            \begin{split}
               I_\HKR(u \cap \beta)&=\tr_{\sheaf F} [\exp(\alpha)(\sheaf F) \circ v \circ \exp(\alpha)(\sheaf
               F)]\\
                 &= \tr_{\sheaf F} [\exp(\alpha)(\sheaf F) \circ v] \\
                & = v \wedge \tr_{\sheaf F} [\exp(\alpha)(\sheaf F)]\\&  = v \wedge \ch(\sheaf F) \\
                & = v \inner \ch(\sheaf F).
            \end{split}
        \]
        Furthermore, $v = a^{-1} \inner v$.
        Thus, we have
        \[
            \begin{split}
                I_\Ko(u \cap \beta) & = \sah \wedge I_\HKR(u \cap \beta) \\
                & = \sah \wedge (v \inner \ch(\sheaf F)) \\
                & = v \inner (\sah \wedge \ch(\sheaf F)) \\
                & = (\sahi \inner v) \inner (\sah \wedge \ch(\sheaf F)) \\
                & = I^\Ko(u) \inner I_\Ko(\beta).
            \end{split}
        \]

\noindent
{\bf The second case $v \in H^1(X, \sheaf T_X)$:}
        In this case, $v \inner$ acts as a derivation (with respect to the wedge product). Thus, one has
        \[
            \begin{split}
I_\HKR(u \cap \beta)&=                \tr_{\sheaf F}
[\exp(\alpha)(\sheaf F) \circ v \circ \exp(\alpha)(\sheaf F)]\\
                & = \tr_{\sheaf F} [\exp(\alpha)(\sheaf F) \circ v \circ \alpha(\sheaf F)] \\
                & = \tr_{\sheaf F} [\exp(\alpha)(\sheaf F) \circ (v \inner \alpha(\sheaf F))] \\
                & = \tr_{\sheaf F} [v \inner \exp(\alpha)(\sheaf F)] \\
                & = v \inner \ch(\sheaf F).
            \end{split}
        \]
 Since the $(1,1)$-part of $\sahi$ is trivial, one clearly has $v = \sahi \inner
 v$. (In fact, an argument similar to the following shows that
 $v=b\inner v$ for any $b=1+\lambda{\rm c}_1(X)+\ldots$.)
 Furthermore, since $a=\hat A^{\frac 1 2}$ stays of pure type under any deformation of
 $X$,  also $v \inner \sah = 0$ (this follows also from Lemma~\ref{lem:second}). Hence
        \[
            \begin{split}
                I_\Ko(u \cap \beta)
                & = \sah \wedge I_\HKR(u \cap \beta) =\sah\wedge(v\inner\ch(\sheaf F))\\
                & = v \inner (\sah \wedge \ch(\sheaf F)) - (v \inner \sah) \wedge \ch(\sheaf F) \\
                & = (\sahi \inner v) \inner (\sah \wedge \ch(\sheaf F)) \\
                & = I^\Ko(v) \inner I_\Ko(\beta).
            \end{split}
        \]

\noindent
{\bf The final case $v \in H^0(X, \bigwedge^2 \sheaf
T_X)$:}
        By Lemma \ref{lem:hk} below, the right hand side
        of~\eqref{eq:cal} is $\sahi \wedge ((\sahi \inner v) \inner (\sah \wedge \ch(\sheaf F))$. Then we have
        \[
            \begin{split}
                I_\Ko(u \cap \beta) & = \sah \wedge I_\HKR(u \cap \beta) \\
                & = (\sahi \inner v) \inner (\sah \wedge \ch(\sheaf F)) \\
                & = I^K(u) \inner I_\Ko(\beta).
            \end{split}
        \]
    \end{proof}

    \begin{lemma}
        \label{lem:hk} Suppose $X$ is irreducible holomorphic
        symplectic.
        Let $v \in H^0(X, \bigwedge^2 \sheaf T_X) = \Hom(\Omega^2_X, \sheaf O_X)$ be
        a holomorphic bivector field and $\sheaf F \in \DCb(X)$. Then
        \[
            \tr_{\sheaf F} [\exp(\alpha)(\sheaf F) \circ v \circ \exp(\alpha)(\sheaf F)] =
            \sahi \wedge ((\sahi \inner v) \inner (\sah \wedge \ch(\sheaf F))).
        \]
    \end{lemma}
This lemma will be proven by means of graph homology in Appendix
\ref{app2}.


\section{Using $\HH^1$}

    The proof of the following result is a straightforward application of the invariance of the first Hochschild cohomology
    under derived equivalence.

    \begin{proposition}\label{mainprop:AV}
        Suppose $A$ is an abelian variety. If for a smooth projective variety $X$ there exists an exact equivalence
        $\DCb(A) \cong \DCb(X)$, then $X$ is as well an abelian variety.
    \end{proposition}

    \begin{proof}
        The order of the canonical bundle and the dimension of a smooth projective variety are derived invariants
        (see e.g.~\cite[Prop.~4.1]{huybrechts:fm}). Hence, $\omega_X$ is trivial and $\dim(X) = \dim(A)$.

        By Theorem~\ref{thm:decomp}, there exists a finite \'etale cover $\pi\colon \widehat X \to X$
        with
        $\widehat X \cong
        B \times \prod_i Y_i \times \prod_j Z_j$ such that $B$ is an abelian
        variety, the $Y_i$ are irreducible holomorphic symplectic, and the $Z_j$
       are Calabi--Yau manifolds of dimension at least three.
        For the latter two types, the first Hochschild
        cohomology $\HH^1 \simeq H^0(\sheaf T) \oplus H^1(\sheaf O)$ is trivial.
        Thus the pull-back yields an injection
        \[
            \pi^*\colon \HH^1(X) \hookrightarrow \HH^1(B).
        \]
        On the other hand, the equivalence $\DCb(A) \simeq \DCb(X)$ induces an isomorphism $\HH^1(A) \cong \HH^1(X)$
        of vector spaces of dimension $2 \dim(A) = 2 \dim(X)$. Hence $2\dim(X) \leq 2 \dim(B)$.
        Therefore, $\widehat X = B$ and $\pi^*\colon \HH^*(X) \to \HH^*(B)$ must be bijective.

        Moreover, since the pull-back $\pi^*\colon \HH^1(X) \to \HH^1(B)$ respects the direct sum decomposition,
        both spaces $H^0(X, \sheaf T_X)$ and $H^1(X, \sheaf O_X)$ are of dimension $\dim(X)$. By
        functoriality of the Albanese morphism, the composition $B \to X \to \Alb(X)$ is
\'etale. Hence $X \to \Alb(X)$
        is étale as well and, therefore, $X$ is an abelian variety.
    \end{proof}

    \begin{remark}
        Note that there exist indeed \'etale quotients of abelian varieties with trivial canonical bundle
        which are not abelian varieties themselves (see~\cite[16.16]{ueno}). Those are however, due to the proposition,
        never derived equivalent to an abelian variety. This also explains why we had to use the Albanese map in
        the proof.
    \end{remark}

    \begin{remark}
        \label{rmk:sc}
        If in the proposition, one furthermore assumes that $X$ is an
        irreducible projective symplectic or a Calabi--Yau manifold of dimension at least three,
        then the contradiction is obtained more easily by using $\HH^1(A) \cong \HH^1(X)$ and the fact that
        $\HH^1(X) = H^0(X, \sheaf T_X) \oplus H^1(X, \sheaf O_X)$ is trivial for any such an $X$.
    \end{remark}

    \begin{remark}
        Alternatively to Remark \ref{rmk:sc} one could use the existence of spherical and projective objects
        (see~\cite{huybrechts:projective})
        on Calabi--Yau
        respectively irreducible symplectic manifolds (provided e.g.~by any line bundle),
        which do not exist on abelian varieties of dimension at least two.
    \end{remark}

    \section{Using $\HH^2$}

    The idea to prove that an irreducible holomorphic symplectic
manifold $Y$ can never be derived equivalent to a Calabi--Yau
manifold $Z$ of dimension at least three  goes as follows: The
infinitesimal deformations (commutative and non-commutative)
parametrised by the second Hochschild cohomology are identified
under any derived equivalence $\DCb(Y)\cong\DCb(Z)$,
i.e.~$\HH^2(Y) \cong \HH^2(Z)$. Algebraic classes $\alpha \in
\HO_0(Y)$ that stay algebraic under all these deformations are of
a very special form (see Proposition~\ref{prop:deforming}). In
particular, they have non-trivial square with respect to the Mukai
pairing. On the other hand, a skyscraper sheaf $k(z) \in \DCb(Z)$
does deform infinitesimally in all deformation directions
parametrised by $\HH^2(Z)$ (which are all commutative for $Z$ a
Calabi--Yau manifold), but its Mukai square is trivial.


\medskip

We shall use  the second Hochschild cohomology and $\cap:\HH^2(X)
\otimes \HH_0(X) \to \HH_{-2}(X)$ instead of the first Hochschild
cohomology as in the case of abelian
    varieties. Consider the two subspaces
    \[
        A(X) := \Set{\alpha \in \HH_0(X) \mid \forall v \in \HH^2(X): v \cap \alpha = 0}
    \]
    and
    \[
        R(X) := \Set{\alpha \in \HO_0(X) \mid
        \forall v \in \HT^2(X): v \inner \alpha = 0}.
    \]

    \begin{remark}
        The contraction with $v \in H^1(X, \sheaf T_X)$ of a cohomology class $\alpha \in H^{p, p}(X)$ measures
the change of the bidegree of $\alpha$ under the infinitesimal
deformation of $X$ corresponding to $v$ (Griffiths
transversality). In analogy, one should think of $R(X)$ as the set
of classes of pure type which stay of
        pure type under all infinitesimal deformations parametrised by the second Hochschild cohomology $\HH^2(X)$,
        which include the classical ones in $H^1(X, \sheaf T_X)$ as well as non-commutative and twisted ones in
        $H^0(X, \bigwedge^2 \sheaf T_X)$ and $H^2(X, \sheaf O_X)$, respectively.
    \end{remark}

    \begin{lemma}\label{lem:IKAR} Suppose $X$ has trivial canonical bundle. Then
        the isomorphism $I_\Ko\colon \HH_0(X) \stackrel{\sim}{\to} \HO_0(X)$ induces an isomorphism
        \[
           I_\Ko\colon A(X) \stackrel{\sim}{\longrightarrow} R(X).
        \]
        Moreover, if $\Phi:\DCb(X)\stackrel{\sim}{\to}\DCb(X')$ is an exact
        equivalence, then $\Phi^{\HH_*}$ defines an isomorphism
        $A(X)\cong A(X')$.
    \end{lemma}

    \begin{proof}
        This follows from the compatibility of the modified HKR-isomorphism with the module
        structures,
        see Corollary \ref{cor:Calconj}.

        The second assertion
    is a consequence of  the general fact that any exact equivalence
        induces an isomorphism of the Hochschild structure.
    \end{proof}

The above proof relies on Corollary \ref{cor:Calconj}, which in
turn uses \cite{vdbergh} and \cite{ramadoss}. If instead one
prefers to use Proposition \ref{prop:secondGraph}, one would get
the following result, which on the one hand works only for
algebraic classes, but, on the other hand, covers manifolds with
non-trivial canonical bundle.

\begin{lemma}\label{lem:IKARweak} Suppose that either
$H^0(X,\bigwedge^2\sheaf T_X)=0$ or that $X$ is irreducible
holomorphic symplectic. Let $\alpha\in \HH_0(X)$  be a class such
that $I_\Ko(\alpha)\in\HO_0(X)$ is algebraic. Then $\alpha\in
A(X)$ if and only if $I_\Ko(\alpha)\in R(X)$.
\end{lemma}

\begin{proof}
By Proposition \ref{prop:secondGraph} we have
$I_\Ko(v\cap\alpha)=I^\Ko(v)\inner I_\Ko(\alpha)$ for all $v\in
\HT^2(X)$, which suffices to conclude.
\end{proof}

    \begin{proposition}
        \label{prop:cy}
        Let $Z$ be a projective manifold of dimension $m$ with
$H^0(Z, \bigwedge^2 \sheaf T_Z) = 0$.
        Then $H^{m, m}(Z) \subset R(Z)$.

        If, moreover, $Z$ is a Calabi--Yau manifold  of dimension $m \ge 3$, then the following holds:
        \[
          (H^{0, 0} \oplus H^{1, 1} \oplus H^{m - 1, m - 1} \oplus H^{m, m})(Z) \subset R(Z).
        \]
    \end{proposition}

    \begin{proof}
        The action of $HT^2(Z) \cong H^1(X, \sheaf T_X) \oplus H^2(X, \sheaf O_X)$ on $H^{m, m}(Z)$ is trivial
        for degree reasons. Thus, $H^{m, m}(Z) \subset R(Z)$.

Let now in addition $Z$ be Calabi--Yau of dimension at least
three. Thus, not only $H^0(Z, \bigwedge^2 \sheaf T_Z) = 0$ but
also  $H^2(Z, \sheaf O_Z) = 0$. Hence it suffices to show that the
contraction with classes in $H^1(Z, \sheaf T_Z) $ annihilates
$H^{q, q}(Z)$ for $q = 0, 1, m - 1, m$. This is trivial for $q =
0, m$. For $q = 1$ and $q = m - 1$ it follows from $H^{0, 2}(Z)
=0$ respectively $H^{m - 2, m}(Z) = (H^{2, 0}(Z))\dual = 0$ (use
Serre duality).
    \end{proof}

    \begin{proposition}
        \label{prop:deforming}
        Suppose $Y$ is a simply-connected, projective symplectic manifold of dimension $2n$. Then $R(Y)$ is the set of all
        classes $\alpha \in H^n(Y, \Omega_Y^n)$ satisfying
        \begin{enumerate}
            \item $\bar \sigma \wedge \alpha = 0$,
            \item $\Lambda_\sigma \alpha = 0$ and
            \item $\alpha \in H^n(Y', \Omega_{Y'}^n)$ for all deformations $Y'$ of $Y$
        \end{enumerate}
        for all holomorphic symplectic forms $\sigma \in H^{2, 0}(Z)$.
        Here, $\Lambda_\sigma$ denotes the dual Lefschetz operator to $L_\sigma := \sigma \wedge \cdot$
        (cf.\ \cite{fujiki}).
    \end{proposition}

    \begin{proof}
        As we later only use that any class in $R(Y)$ satisfies (1)--(3), we only prove this inclusion.
        The other relies on the same arguments and is left to the
        reader. For the general theory of hyper-K\"ahler manifolds
        we refer to \cite{fujiki} and \cite{ghj}.

        By the Decomposition Theorem~\ref{thm:decomp}, we have $Y = \prod_i Y_i$,
        where each factor $Y_i$ is an irreducible projective symplectic manifold. Via the isomorphism $I^\Ko$, one has
        \begin{equation}
            \label{eq:hh2onhk}
            \HH^2(Y) \cong \bigoplus_i \left(H^2(Y_i, \sheaf O_{Y_i}) \oplus H^1(Y_i, \sheaf T_{Y_i})
            \oplus H^0(Y_i, \bigwedge\nolimits^2 \sheaf
            T_{Y_i})\right)
        \end{equation}
        with $H^2(X, \sheaf O_{Y_i})$ spanned by the anti-holomorphic symplectic form $\bar\sigma_i$ on $Y_i$,
        and $H^0(Y_i, \bigwedge^2 \sheaf T_{Y_i})$ identified with $H^0(Y_i, \sheaf O_{Y_i})$ by contraction
        with $\sigma_i$.

        Consider $\alpha \in R(Y) \subset \HO_*(Y)$. As the
        direct sum decomposition~\eqref{eq:hh2onhk} respects the decomposition
        $\HO_*(Y) \cong \bigotimes_i \HO_*(Y_i)$, it follows that
        $\alpha \in \bigotimes_i R(Y_i)$. Thus we may assume that $Y = Y_i$.

        Since $\alpha \in R(Y)$, we have $\bar\sigma \inner \alpha = 0$.
        Standard Lefschetz theory of the anti-holomorphic
        form $\bar \sigma$ on $Y$ then shows $\alpha \in\bigoplus_{p\geq n} H^{p}(Y, \Omega_Y^{p})$.

        Contraction with the canonical generator $v \in H^0(Y, \bigwedge^2 \sheaf T_Y)$
        is given by the dual Lefschetz operator to $\sigma$,
        i.e.~$v \inner \alpha = \Lambda_{\sigma}\alpha = 0$. Again by Lefschetz theory, this means that
        $\alpha \in \bigoplus_{p\le n}H^{p}(Y, \Omega_Y^{p})$ and together with
        the above result this yields
        $\alpha \in H^n(Y, \Omega_Y^n)$.

        By definition of $R(Y)$, we know that $v \inner \alpha = 0$ for all $v \in H^1(Y, \sheaf T_Y)$. Thus
        infinitesimally, $\alpha$ stays of type $(n, n)$. By the general theory of hyper-K\"ahler manifolds, this also
        shows that $\alpha$ stays of type $(n, n)$ on any hyper-Kähler rotation of $Y$. As any two
        deformations are connected by twistor lines, it follows that $\alpha$ is of type $(n, n)$ on any deformation
        of $Y$.
    \end{proof}

    \begin{remark}
        In other words, the only classes $\alpha$
        of pure type on a simply-connected holomorphic symplectic manifold $Y$ that
        stay of pure type under any infinitesimal deformation in $\HH^2(Y)$
        are the classes in $H^n(Y, \Omega_Y^n)$ that are
        primitive with respect to all holomorphic
        symplectic forms and their anti-holomorphic conjugates, and which stay of
        type $(n, n)$ under any classical deformation.
    \end{remark}

    \begin{remark}
        Note that for rational cohomology classes in $H^{n}(Y, \Omega_Y^n)$,
        so in particular for algebraic ones, conditions (1)
        and (2) are equivalent. Also note that a class satisfying (3) need not necessarily
        be primitive, e.g.~$c_2(Y)$ on a four-dimensional $Y$ is not.
    \end{remark}

    \begin{lemma}
        \label{lem:qform}
        Let $X$ be a smooth projective variety such that there exists an exact equivalence $\DCb(X) \cong \DCb(Y)$, where
        $Y$ is an irreducible projective symplectic manifold of dimension $2n$. Then there exists a non-degenerate
        quadratic form
        \[
            q\colon \HH^2(X) \to \sqrt[n]{\HH^{2n}(X)}
        \]
        such that
        $q(\alpha)^n = \alpha^{2n}$ for all $\alpha \in \HH^2(X)$ and such that the subring of $\HH^*(X)$ generated
        by $\HH^2(X)$ is given by
        \[
            \Sym\HH^2(X) = \Sym^*\HH^2(X)/\langle\alpha^{n + 1}\mid q(\alpha) = 0\rangle.
        \]
    \end{lemma}

    \begin{proof}
        One has the following isomorphisms of graded rings:
        \[
            \HH^*(X) \cong \HH^*(Y) \cong \HT^*(Y) = H^*(Y, \bigwedge\nolimits^* \sheaf T_Y) \cong H^*(Y, \Omega_Y^*),
        \]
        where the first isomorphism is induced by the exact equivalence $\DCb(X) \cong \DCb(Y)$,
        the second isomorphism is the modified HKR-isomorphism $I^\Ko$,
        and the third isomorphism comes from the isomorpism $\sheaf T_Y \cong \Omega_Y$ induced by
        the symplectic form on $Y$.

        By~\cite{beauville:kummer, fujiki}, the map $\alpha \mapsto \alpha^{2n}$ on $H^2(Y, {\mathbb C})$ possesses an $n$-th root,
        namely the Beauville--Bogomolov quadratic form $q_Y$ of $Y$. By a result of Verbitsky (see
        \cite{bogomolov}), one has
        \[
            \Sym H^2(Y, {\mathbb C}) = \Sym^*H^2(Y, {\mathbb C})/\langle\alpha^{n + 1}\mid q(\alpha) = 0\rangle.
        \]
        Since $\HH^*(X) \cong H^*(Y, \Omega_Y^*)$, the claim of the Lemma follows.
    \end{proof}

    \begin{theorem}\label{mainthm:HK}
        Let $Y$ be an irreducible projective symplectic manifold. If for a smooth projective variety $X$ there
        exists an exact equivalence $\Phi\colon \DCb(Y) \cong \DCb(X)$, then $X$ is also an irreducible
        projective symplectic manifold.
    \end{theorem}

    \begin{proof}
        As the order of the canonical bundle and the dimension are  derived invariants (see e.g.\ \cite[Prop.\ 4.1]{huybrechts:fm}), we have
         $\sheaf \omega_X \cong \sheaf O_X$ and $\dim X = \dim Y =: 2n$.

        Assume that $X$ does not possess a symplectic form. Then for any
        $v \in H^0(X, \bigwedge^2 \sheaf T_X)$ one would have  $v^n = 0$.
        By the description of $\Sym\HH^2(X)$ given in Lemma~\ref{lem:qform}, this would show that $v = 0$.
        Thus, $H^0(X, \bigwedge^2 \sheaf T_X) = 0$.

        Consider a closed point $x \in X$ and its structure sheaf $k(x)$. Its Mukai vector
        $\ch(k(x)) $ is the generator of $H^{2n}(X,
        \Omega_X^{2n})$, which by Proposition~\ref{prop:cy}  is contained in $R(X)$. As $\Phi$ is
        an equivalence, there exists a complex $\sheaf F_x \in \DCb(Y)$ with $\Phi(\sheaf F_x) \cong k(x)$.
It follows that $\gamma := \ch(\sheaf F_x)\wedge
\sqrt{\td(Y)}={\rm ch}(\sheaf F_x)\wedge a \in R(Y)$. Indeed, by
commutativity of~\eqref{eq:diach} and by Lemma \ref{lem:IKAR} or
\ref{lem:IKARweak}, one has
\begin{eqnarray*}
\Phi^{\HH_*}(I_\Ko^{-1}(\gamma))&=&\Phi^{\HH_*}(I_\HKR^{-1}(\ch(\sheaf
F_x)))=I_\HKR^{-1}(\ch(k(x)))\\&=&I_K^{-1}(\ch(k(x)))\in A(X)
\end{eqnarray*} and hence $I_\Ko^{-1}(\gamma)\in A(Y)$, which is equivalent
to $\gamma\in R(Y)$ (again by Lemma \ref{lem:IKAR} or
\ref{lem:IKARweak}).
        Due to
        Proposition~\ref{prop:deforming}, one thus has $\gamma \in H^n(Y, \Omega_Y^n)$.

        Since $\Phi$ is an equivalence, $\chi(\sheaf F_x, \sheaf F_x) = \chi(k(x), k(x)) = 0$.
        This leads to a contradiction as follows: On the one hand,
        \[
            0 = \chi(\sheaf F_x, \sheaf F_x) =(-1)^n\int_Y\gamma\wedge\gamma.
        \]
        (For the last  equality use  that $\gamma\in
        H^{n}(Y,\Omega_Y^n)$.)
        On the other hand, again due to Proposition~\ref{prop:deforming}, one knows that $\gamma \in H^n(Y, \Omega_Y^n)$
        is also of type $(n, n)$ on any hyper-K\"ahler rotation $Y'$ defined by a complex structure say $J$. Now write
        $\sigma = \omega_J + i \omega_K$, where $\omega_J$ is the K\"ahler form on $Y'$. Using
        $\bar\sigma \wedge \gamma = 0$ of Proposition~\ref{prop:deforming} and the fact that $\gamma$ as an algebraic Chern
        class is real, we find that $\omega_J \wedge \gamma = 0$, i.e.~$\gamma$ is an $\omega_J$-primitive $(n, n)$-class
        on $Y'$. Thus by the Hodge--Riemann bilinear relations
        \[
            \int_Y \gamma \wedge \gamma \neq 0,
        \]
        a contradiction. Thus our assumption has to be wrong and therefore
        $X$ admits a symplectic form.

        It remains to show that $X$ is irreducible symplectic.
        As
        \[
            \Sym\HH^2(X) = \Sym^*\HH^2(X)/\langle\alpha^{n + 1}\mid q(\alpha) = 0\rangle,
        \]
        it follows that the assumptions of Proposition~\ref{prop:ihs} in the appendix about the spaces $H^{k, 0}(X)$ are fulfilled, which
        finally proves the theorem.
    \end{proof}

    \begin{corollary}
        An irreducible projective symplectic manifold can never be derived equivalent to a Calabi--Yau manifold
        of dimension at least three.
        \qed
    \end{corollary}

    As another step towards answering our general Question~\ref{qu:general}, we prove the following partial result:
    \begin{proposition}
        Let $Y_1, \ldots, Y_n$ and $Y'_1, \ldots, Y'_m$ be irreducible
        projective symplectic varieties.
        Set $Y := \prod_i Y_i$ and $Y' := \prod_j Y_j'$.
        Assume that there exists an exact
        equivalence $\DCb(Y) \cong \DCb(Y')$. Then there exists a bijection
        $\sigma\colon \Set{1, \ldots, n} \to \Set{1, \ldots, m}$ with $\dim(Y_i) = \dim(Y'_{\sigma(i)})$ and
        $b_2(Y_i) = b_2(Y'_{\sigma(j)})$.
    \end{proposition}

    \begin{proof}
        Consider the variety
        \[
            Q := \Set{[\alpha] \mid \alpha^{\dim Y} = 0} \subset {\mathbb P}(\HH^2(Y)).
        \]
        By Lemma~\ref{lem:qform}, the irreducible components of $Q$ are in bijection to the irreducible factors $Y_i$ of $Y$.
        Furthermore, each irreducible component $Q_i$ is a quadric whose rank is the second Betti number of
        the corresponding irreducible factor. As $Q$ is a derived invariant, the claim of the Proposition follows.
    \end{proof}

 \begin{remark}
For the time being we cannot exclude that a Calabi--Yau manifold
$Z$ is derived equivalent to a smooth projective variety $X$ which
itself is not Calabi--Yau (in our restrictive sense). E.g.\ we can
neither exclude that $X$ is a finite \'etale quotient of a
Calabi--Yau manifold nor that it is a product of Calabi--Yau
manifolds with possibly an irreducible holomorphic symplectic
factor. An abelian factor can easily be excluded by studying the
first Hochschild cohomology.
 \end{remark}
    \appendix

    \section{Irreducible holomorphic symplectic manifolds}

The following result might be known to the experts, but we were
unable to find a reference for it.
    \begin{proposition}
        \label{prop:ihs}
        Let $X$ be a holomorphic symplectic manifold of complex dimension $2n$
        such that $H^{k, 0}(X) \cong \CC$ for $k = 0, 2, \ldots, 2n$ and $H^{k, 0}(X) = 0$ otherwise.

        Then $X$ is simply-connected, i.e.~$X$ is an irreducible holomorphic symplectic manifold.
    \end{proposition}

    \begin{proof}
        Let $\pi\colon \widehat X \to X$ be a cover of $X$ as in the Decomposition Theorem~\ref{thm:decomp}
        and let  $d$ be its degree.

        Assume that $\widehat X$ contains an abelian factor. In this case, the fundamental group of $X$ would contain an infinite
        free abelian group, which contradicts the assumption that $H^{1, 0}(X) = 0$.
        Thus, $\widehat X$ is simply-connected.

        Endow $X$ with a hyper-K\"ahler metric $g$,
        which induces a hyper-K\"ahler metric also on $\widehat X$. By the Decomposition Theorem~\ref{thm:decomp},
        $\widehat X$ splits into a product $Y_1 \times \ldots \times Y_k$ of irreducible hyper-K\"ahler manifolds.
        Set $n_i := \frac 1 2 \dim Y_i$.

        The holomorphic Euler characteristic of $\widehat X$ is given by
        \[
            \chi(\widehat X, \sheaf O_{\widehat X}) = \chi(Y_1) \cdots \chi(Y_k) = (1 + n_1) \cdots (1 + n_k),
        \]
        while the holomorphic Euler characteristic of $X$ is by assumption
        \[
            \chi(X, \sheaf O_X) = 1 + n.
        \]
        As $\chi(\widehat X, \sheaf O_{\widehat X}) = d \cdot \chi(X, \sheaf O_X)$, one has the following
        equality:
        \begin{equation}
            \label{eq:chi}
            d \cdot (1 + n) = (1 + n_1) \cdots (1 + n_k).
        \end{equation}

        Fix a point $p \in \widehat X$ and let $q:=\pi(p)$. The splitting $\widehat X = Y_1 \times \cdots \times Y_k$ induces
        a splitting $\sheaf T_X(p) \cong \CC^{2 n_1} \oplus \cdots \oplus \CC^{2 n_k}$ of the
        holomorphic tangent space at $p$ as a unitary vector space.
        We can choose the isomorphism in such a way that under it the holonomy group of $\widehat X$ at $p$ is
        identified with $H^0 := \Sp(n_1) \times \cdots \times \Sp(n_k)$ with
        its natural action.

        Let $H$ be the holonomy group of $X$ at $q$
        acting unitarily on $\CC^{2n} = \CC^{2 n_1} \oplus \cdots \oplus
\CC^{2 n_k}$. Then $H/H^0$ is a certain quotient group of
$\pi_1(X, q)$.

        By the holonomy principle, the space $H^{2, 0}(X) = H^0(X, \Omega^2_X)$ is given by the invariants of
        $\bigwedge^2 \CC^{2n}$ under the action of $H$. By assumption,
$H^{2, 0}(X)$ is spanned by a unique (up to scaling) holomorphic
symplectic form $\sigma$ and its pull-back can be written as
$\pi^* \sigma = \sigma_1 + \cdots + \sigma_k$
        with $\sigma_i$ a symplectic form on $Y_i$. We can choose the isomorphism
        $\sheaf T_X(p) \cong \CC^{2 n}$ from above such that $\pi^* \sigma(p)$ becomes the
        standard symplectic form on $\CC^{2n}$. In particular, $\sigma_i(p)$ becomes the standard symplectic
        form on the summand $\CC^{2 n_i}$, which we denote again by $\sigma_i$. In this way, we realize
         $H$ as a subgroup of the standard $\Sp(n)$.

         As $H^0$ is a normal subgroup of $H$,
        it follows that $H$ is a subgroup of the normaliser $N$ of $H^0$ in $\Sp(n)$.
        By Lemma~\ref{lem:norm} below, there is a group homomorphism $\rho\colon N \to \SG k$ such that
        $A \sigma_i = \sigma_{\rho(A)(i)}$ for all $i = 1, \ldots, k$ for $A \in N$. Furthermore, one clearly has $n_i = n_{\rho(A)(i)}$.
        Denote the image of $H$ under $\rho$ by $\mathfrak S$.

        Since $H^{2, 0}(X)$ is one-dimensional, the holonomy
        principle yields that
        the space of invariants of $\bigwedge^2 \CC^{2n}$ under the action of $H$ is
         one-dimensional as well. On the other hand, the space of invariants under the action of $H^0$ is given by
        elements of the form $\lambda_1 \sigma_1 + \cdots + \lambda_k \sigma_k$ with $\lambda_i \in \CC$ arbitrary.

        Thus $\mathfrak S$ has to act transitively on the $\sigma_1, \ldots, \sigma_k$.
        It follows that $m := n_1 = \ldots = n_k$, i.e.~all factors are of the same dimension. Furthermore,
        the order of $\mathfrak S$ has to be a multiple of $k$.

        Assume that $m > 1$. We may further
        assume that $k > 1$, as otherwise $d = 1$ by~\eqref{eq:chi} and $X$ would
        already  be simply-connected.
        Then the space of invariants of $\bigwedge^4 \CC^{2n}$ under the action of $N$
        (and thus $H$)
        includes the two-dimensional subspace spanned by
        $\sum_i \sigma_i^2$ and $ \sum_{i < j} \sigma_i \sigma_j$. Hence $\dim (H^{4, 0}(X)) > 1$
        by the holonomy principle, contradicting the assumption. Thus  only the case $m = 1$ remains.

        By construction,  $\mathfrak S$ is a subquotient of $\pi_1(X, q)$ and, therefore,
        its order divides  $d=|\pi_1(X, q)|$.
        In particular, $d$ has to be a multiple of $k$, say $d = e \cdot k$ for a natural number $e$.
        Together with \eqref{eq:chi} this yields the equation (with $n_i = m = 1$)
        \[
            e\cdot k\cdot (1 + k) = 2^k
        \]
        having  $e = k = 1$ as its unique solution. Thus $X$
        is already simply-connected.
    \end{proof}

    \begin{lemma}
        \label{lem:norm}
        Let $n = n_1 + \ldots + n_k$ be a partition of a natural number.
        Consider $\CC^{2n} = \CC^{2 n_1} \oplus \cdots \oplus \CC^{2 n_k}$ with its standard
        hermitian complex-symplectic structure. Denote the standard symplectic form on the summand
        $\CC^{2n_i}$ by $\sigma_i$.

        Let $N$ be the normaliser of
        $G := \Sp(n_1) \times \cdots \times \Sp(n_k)$ in $\Sp(n)$.
        The natural action of $\Sp(n)$ induces a natural action of $N$ on $\CC^{2n}$ and $\bigwedge^2 \CC^{2n}$.

        Then there is a group homomorphism
        $\rho\colon N \to \SG k$ such that $A \sigma_i = \sigma_{\rho(A)(i)}$
        for all $A \in N$ and $i = 1, \ldots, k$.

        Furthermore, the image of $\rho$ is a subgroup containing only permutations $\tau$ with
        $n_{\tau(i)} = n_i$ for all $i = 1, \ldots, k$.
    \end{lemma}

    \begin{proof}
        As $N$ is the normaliser of $G$, it maps the space of $G$-invariants  into itself. The space of
        invariants of $G$ is spanned by $\sigma_1$, \ldots, $\sigma_k$.

        Let $A \in N$ and pick $i$ such that $\sigma_i$ is of minimal rank
        (i.e.~$n_i$ is minimal among the $n_1, \ldots, n_k$). By the same
         reasoning as in the proof of the Proposition, we have $A \sigma_i = \sum_j
        \lambda^j_i \sigma_j$ for certain $\lambda^j_i \in \CC$.
        The rank of a two-form does not change
        under a general linear transformation. Thus
        the sum in $A \sigma_i$ can only consist of one term, i.e.~
        we have to have $A \sigma_i = \lambda_i \sigma_{\rho(A)(i)}$ with $n_i = n_{\rho(A)(i)}$
        for a certain $\rho(A)(i)$ and $\lambda_i\in \CC^*$.

        In this way, we define $\rho(A)(i)$ for all $i$ with minimal $n_i$. As $A$ is invertible,
        $\rho(A)$ must be a permutation of those $i$ for which $n_i$ is
        minimal.

        Suppose we have shown that all $A\in N$ permute (up to scaling) the forms
        $\sigma_\ell$ of rank $2n_\ell<r$ and let $\sigma_i$ be of rank
        $r$. Then $A \sigma_i = \sum_j \lambda^j_i \sigma_j$ for certain $\lambda^j_i \in \CC^* $.
        Clearly, $r=\rk (\sigma_i)=\rk(A\sigma_i)=\sum_{\lambda_j\ne0}\rk(\sigma_j)$. Thus, either
        $A\sigma_i= \lambda_i^j\sigma_j$ for some $j$ or $A\sigma_i$ is a linear combination of symplectic forms
        $\sigma_j$ of strictly  smaller rank. As $A$ is invertible, the latter would contradict the assumption
        that $A^{-1}$ permutes all $\sigma_\ell$ of rank $<r$. Hence $A\sigma_i=\lambda_i^j\sigma_j=\lambda_i\sigma_{\rho(A)(i)}$
         and $\rk(\sigma_i)=\rk(\sigma_{\rho(A)(i)})$, i.e.\ $n_i=n_{\rho(A)(i)}$. Hence, $A$ permutes
forms $\sigma_\ell$ of rank $\leq r$ up to  scaling by
$\lambda_\ell\in\CC^*$.

Continuing in this way, we obtain a permutation $\rho(A)$ of
$\Set{1, \ldots, k}$ with
        $n_{\rho(A)(i)} = n_i$ and such that
        $A\sigma_i = \lambda_i \sigma_{\rho(A)(i)}$ for certain $\lambda_i\in\CC^*$.
        It is clear that $A \mapsto \rho(A)$ is a group homomorphism.

        Now $\sigma = \sigma_1 + \cdots + \sigma_k$ is  invariant under the action of the group $\Sp(n)$, thus also
        under the normaliser $N$, and hence
        $\lambda_1 = \ldots = \lambda_k = 1$, which proves the Lemma.
    \end{proof}

\section{Proof of Lemma \ref{lem:hk}}\label{app2}

    \subsection{Graph homology}

    We need to introduce a bit of notation from graph homology. We follow mostly~\cite{thurston} and refer the reader
    it
    for details:

    A \emph{Jacobi diagram} is a uni-trivalent graph with a chosen cyclic orientation of the set of half-edges at each
    trivalent vertex. By
    \[
        \mathcal A(\astar_{x_1} \cdots \astar_{x_k} \, \orint_{y_1} \cdots \orint_{y_l} \, \orcir_{z_1} \cdots
        \orcir_{z_m})
    \]
    we denote the (suitably completed, see~\cite[Sect.\ 2.5, 3.3]{thurston})
    ${\mathbb Q}$-vector space spanned by all Jacobi diagrams whose set of univalent vertices is partitioned into sets
    labelled by
    $x_1, \ldots, x_k$, linearly ordered sets labelled by $y_1, \ldots, y_l$ and cyclicly ordered sets labelled by $z_1, \ldots, z_m$
    modulo the AS, IHX and STU relations (for the definition of the AS, IHX and STU
    relations, see~\cite{thurston}). Elements of these spaces are called \emph{graph homology classes}.
    In what follows, let $X$ and $X'$ denote disjoint collections of disjointly
    labelled stars, oriented intervals, and oriented circles.

    \begin{example}
        By $_x\astrut_y \in \mathcal A(\astar_x \, \astar_y)$ we denote the graph homology class defined by the Jacobi diagram
        consisting of two univalent vertices labelled with $x$ and $y$, respectively. We call this graph homology class a
        \emph{strut}.
    \end{example}

    There are commutative, associative products
    \begin{eqnarray*}
        \mathcal A(\astar_{x_1} \cdots \astar_{x_k} \, X) \otimes \mathcal A(\astar_{x_1} \cdots \astar_{x_k} \, X')
        &\to& \mathcal A(\astar_{x_1} \cdots \astar_{x_k} \, X \, X'),\\
        C \otimes D &\mapsto& C \cup_{x_1 \cdots x_k} D,
    \end{eqnarray*}
    which are given by the disjoint union of the underlying Jacobi diagrams.
    There are further associative products
    \[
        \mathcal A(\orint_x \, X) \otimes \mathcal A(\orint_x \, X') \to \mathcal A(\orint_x \, X \, X'),\quad
        C \otimes D \mapsto C \juxta_x D,
    \]
    which are given by juxtaposition of the orders of the univalent vertices labelled by $x$.
    There is a natural linear map
    \[
        \pi_x\colon \mathcal A(\astar_x \, X) \to \mathcal A(\orint_x \, X),
    \]
    which is given by averaging over all possibilities to linearly order
the univalent vertices labelled by $x$. This map is an isomorphism
of vector spaces (see, e.g.~\cite[Sect.\ 3.3]{thurston}). Thus we
can identify the two spaces.

    There is another natural linear map
    \[
        \tr_x\colon \mathcal A(\orint_x \, X) \to \mathcal A(\orcir_x \, X),
    \]
    which is given by applying the forgetful map that turns the linear order of the univalent vertices labelled by $x$ into
    a cyclic order. By~\cite{thurston}, this map is an isomorphism whenever $X$ consists only of oriented circles, i.e.~does not
    include stars or oriented intervals.

    We introduce bilinear forms
    \[
        \mathcal A(\astar_{x_1} \cdots \astar_{x_k} \, X) \otimes
        \mathcal A_0(\astar_{x_1} \cdots \astar_{x_k} \, X') \to \mathcal A_0(X \, X'),\quad
        C \otimes D \mapsto \langle C, D\rangle_{x_1 \cdots x_k}
    \]
    which are given by summing over all ways of gluing all $x_i$-labelled univalent vertices of $C$ pairwise to all
    $x_i$-labelled univalent vertices of $D$, $i = 1 \ldots k$.
    Here $\mathcal A_0(\cdots)$ denotes the subspace of $\mathcal A(\cdots)$ generated
    by those Jacobi diagrams that do not include any component $_{x_i}\astrut_{x_i}$.

    We shall also need an operator that allows us to relabel graph homology classes: Let $\Delta_x^y$ be the linear operator
    that is given by replacing the labels of $y$ by $x$.
    For vertices labelled with stars, this operation can be
    extended: Let
    \[
        \Delta_{x_1 x_2}^y\colon \mathcal A(\astar_y \, X) \to \mathcal A(\astar_{x_1} \, \astar_{x_2} X)
    \]
    be the linear operator whose value is the sum over all ways of replacing the labels $y$ by one of $x_1$ and $x_2$.

    \begin{example}
        One has
        \[
            \langle C, D_1 \cup_x D_2\rangle_x = \langle \Delta_{y_1 y_2}^x C, (\Delta_{y_1}^x D_1) \otimes
                (\Delta_{y_2}^x D_2)\rangle_{y_1 y_2}
        \]
        when either $C$ or $D_1$ and $D_2$ having no struts $_x\astrut_x$.
    \end{example}

    In~\cite{thurston},
    D.~Thurston also introduces diagrammatic differential operators $\partial_C$. Here they are denoted slightly differently:
    Let
    \[
        \mathcal A_0(\astar_x \, X) \otimes \mathcal A(\astar_x \, X') \to \mathcal A(\astar_x \, X \, X'),\quad
        C \otimes D \mapsto C \inner_x D,
    \]
    be the differential operator (with respect to $\cup_x$) that is given by summing over all ways of gluing all $x$-labelled
    univalent vertices of $C$ to some $x$-labelled univalent vertices of $D$.
    (We can also define the inner product if the right hand side has no struts, but the left hand side (possibly) has.)

    \begin{example}
        This differential operator can be expressed in terms of the inner product $\langle\cdot,\cdot\rangle$, namely
        \[
            \Delta^x_y (C \inner_x D) = \langle C \cup_x \exp(_x\astrut_y), D\rangle_x.
        \]
        Here, $\exp$ means the exponential series with respect to the product $\cup_{xy}$.
    \end{example}

    \subsection{Weight systems}

    The reason why we have recalled these facts about graph homology is that graph homology classes define certain cohomology classes
    on holomorphic symplectic manifolds, the so-called~\emph{Rozansky--Witten classes}
    (for a detailed treatment see, e.g.,~\cite{nieper:book} or~\cite{rowi}): Let $C \in \mathcal A(\astar_x \, \orint_y)$
    be a graph homology class which is
    represented by a Jacobi diagram $\Gamma$, whose univalent are labelled by $x$ and $y$, and which has a linear order of the
    univalent vertices labelled by $y$ chosen. Let the number of trivalent vertices be $k$, the number of univalent
    vertices labelled $x$ be $\ell$ and the number of univalent vertices labelled $y$ be $m$.

    Let $\mathfrak g$ be a finite-dimensional
    metric Lie algebra over ${\mathbb C}$, i.e.~a Lie algebra endowed with a scalar product
    $\langle\cdot, \cdot\rangle$ with respect to which the
    adjoint maps are skew-symmetric. Let $E$ be a finite-dimensional $\mathfrak g$-module. Then we can construct an element
    \begin{equation}
        \Gamma(\mathfrak g, \langle\cdot,\cdot\rangle, E) \in \Sym^\ell \mathfrak g \otimes \End(E)
    \end{equation}
    as follows: Use the scalar product on $\mathfrak g$ to identify $\mathfrak g$ and $\mathfrak g\dual$. In particular,
    we can view the Lie bracket as an element in $\mathfrak g^{\otimes 3}$. Take one tensor copy of this element for each of the $k$
    trivalent vertices in $\Gamma$ and one copy of the scalar product as an element
    in $\mathfrak g^{\otimes 2}$ for each edge connecting two univalent vertices.
    Use the scalar product to contract the resulting tensor
    tensor along the edges of
    the graph $\Gamma$ connecting trivalent vertices. This process yields an element in
    $\mathfrak g^{\otimes \ell} \otimes \mathfrak g^{\otimes m}$.
    The action of $\mathfrak g$ on $E$ can be viewed as a map $\mathfrak g \to \End(E)$. Applying this map to the last $m$ tensor
    factors in our element in $\mathfrak g^{\otimes \ell} \otimes g^{\otimes m}$ and then composing the endomorphisms in the order given
    by the linear order of the univalent vertices labelled $y$ yields an element in $\mathfrak g^{\otimes l} \otimes \End(E)$. Finally
    project down to the symmetric tensors to get a well-defined element
    $\Gamma(\mathfrak g, \langle\cdot, \cdot\rangle, E) \in \Sym^\ell \mathfrak g \otimes \End(E)$. Using the properties of the Lie bracket,
    the ad-invariance of the scalar product and the fact that the map $\mathfrak g \to \End(E)$ is a morphism of Lie algebras allows one
    to show that $\Gamma(\mathfrak g, \langle\cdot, \cdot\rangle, E)$ does only depend on the graph homology class $C$ of $\Gamma$
    and not on $\Gamma$ itself. Thus, we have defined an element
    \[
        C(\mathfrak g, \langle\cdot, \cdot\rangle, E) \in \Sym^\ell \mathfrak g \otimes \End(E).
    \]
Note that for $m=0$ we obtain an element in $\Sym^\ell\mathfrak
g\cong S^\ell \mathfrak g \otimes \id_E$.
    Given a graph homology class $D \in \mathcal A(\astar_x \, \orcir_y)$, we set
    \[
        D(\mathfrak g, \langle\cdot, \cdot\rangle, E) := \tr_E C(\mathfrak g, \langle\cdot,
        \cdot\rangle, E) \in \Sym^\ell \mathfrak g,
    \]
    where $C$ is any lift of $D$ under the natural map $\tr_y\colon \mathcal A(\astar_x \, \orint_y) \to \mathcal A(\astar_x \, \orcir_y)$.
    This is a well-defined assignment due to the cyclic invariance of the trace map $\tr_E$ on $\End(E)$.

    This construction applies to holomorphic symplectic manifolds as follows: The (shifted) tangent sheaf $\sheaf T_X[-1]$
    of a holomorphic symplectic manifold $X$ can be viewed as a Lie algebra in the bounded
    derived category of $X$: The Lie bracket
    is given by the Atiyah class
    \[
        A(\sheaf T_X)\colon \sheaf T_X[-1] \otimes \sheaf T_X[-1] \to \sheaf T_X[-1].
    \]
    (In fact, this is true on any complex manifold.)
    A holomorphic symplectic form $\sigma$ endows $\sheaf T_X[-1]$ with an invariant metric (of degree $2$)
    \[
        \sigma\colon \sheaf T_X[-1] \otimes \sheaf T_X[-1] \to \sheaf O_X[-2].
    \]
    Finally, every bounded complex $\sheaf F$ of coherent sheaves becomes a $\sheaf T_X[-1]$-module, where the module action
    is given by the Atiyah class
    \[
        A(\sheaf F)\colon \sheaf T_X[-1] \otimes \sheaf F \to \sheaf F.
    \]
    As the above construction of $C(\mathfrak g, \langle\cdot, \cdot\rangle, E)$ also works in the more general context of Lie algebra
    objects in suitable categories, one can associate
    to each such triple $(X, \sigma, \sheaf F)$ classes
    $$        C(X, \sigma, \sheaf F)  \in H^k(X, \Omega_X^\ell \otimes \sheaf E nd(\sheaf
        F))$$
(resp.\ $C(X, \sigma) = C(X, \sigma, \sheaf F)\in H^k(X,\Omega_X^\ell)$ if $m=0$)
and
$$D(X, \sigma, \sheaf F)  \in H^k(X, \Omega_X^\ell).$$
    (Here, we have identified $\Omega_X$ with $\sheaf T_X$ by means of the symplectic form.)
    The maps $C \mapsto C(X, \sigma, \sheaf F)$ and $D \mapsto D(X, \sigma, \sheaf F)$ are called the \emph{(Rozansky--Witten)
    weight system given by $(X, \sigma, \sheaf F)$}.

    The following examples all follow from the definitions (see also~\cite{nieper:book}).
    \begin{example}
        Let $\Omega_x \in \mathcal A(\astar_x)$ be the so-called Wheeling element (see~\cite[Sect.\ 2.7]{thurston}). Then
        \[
            \Omega_x(X, \sigma) = \sah(X) \in \HO_0(X)
        \]
        (see~\cite{hitchin-sawon}).
    \end{example}

    \begin{example}
        Consider $_x\astrut_x \in \mathcal A(\astar_x)$. One has
        \[
            _x\astrut_x(X, \sigma, \sheaf F) = 2 \sigma.
        \]
    \end{example}

    \begin{example}
        Let $C \in \mathcal A(\astar_x \, X)$ and $D \in \mathcal A(\astar_x)$. Then
        \[
            (C \cup_x D)(X, \sigma, \sheaf F) = C(X, \sigma, \sheaf F) \wedge D(X, \sigma, \sheaf F).
        \]
    \end{example}

    \begin{example}
        Let $C \in \mathcal A(\astar_x \, \orint_y)$ and $D \in \mathcal A(\orint_y)$. Then
        \[
            (C \juxta_y D)(X, \sigma, \sheaf F) = C(X, \sigma, \sheaf F) \circ D(X, \sigma, \sheaf F).
        \]
    \end{example}

    \begin{example}
        Consider $\exp({_x \astrut_y}) \in \mathcal A(\astar_x \, \orint_y)$. Then
        \[
            \exp({_x \astrut_y})(X, \sigma, \sheaf F) = \exp(\alpha)(\sheaf F).
        \]
    \end{example}

    \begin{example}
        Let $C \in \mathcal A(\astar_x)$ and $D \in \mathcal A(\astar_x)$. Then
        \[
            (C \inner_x D)(X, \sigma, \sheaf F) = C(X, \sigma, \sheaf F) \inner_x D(X, \sigma, \sheaf F),
        \]
        where we view $C(X, \sigma, \sheaf F)$ as an element in $\HT^*(X)$ by means of the isomorphism
        $H^p(X,\bigwedge^q\sheaf T_X) \cong H^p(X,\Omega_X^q)$ induced by the symplectic form $\sigma$.
    \end{example}

    \begin{example}
        Let $\pi\colon \mathcal A(\astar_x \, X) \to \mathcal A(\orint_x \, X)$ be the natural map. Then
        \[
            \pi(C)(X, \sigma, \sheaf F) = C(X, \sigma, \sheaf F) \circ \exp(\alpha)(\sheaf F),
        \]
        where we view $C(X, \sigma, \sheaf F)$ as an element in $\HT^*(X)$ by means of the isomorphism
       $H^p(X,\bigwedge^q\sheaf T_X) \cong H^p(X,\Omega_X^q)$ as
       above.
    \end{example}

    \subsection{The proof}

    \begin{proof}[Proof of Lemma~\ref{lem:hk}]
Only the case $v \neq 0$ needs a proof, but then we may assume
that $v$ is dual to the holomorphic symplectic form $\sigma$.

        By Lemma~\ref{lem:wheeling2} below, we have
        \begin{equation}
            \label{eq:cor}
            \tr_x [(\Omega_x \inner_x \exp(_x\astrut_y)) \juxta_x (\Omega_x \inner_x \frac 1 2 {_x\astrut_x})]
            = \tr_x [\Omega_x \inner_x (\exp(_x \astrut_y) \cup_x \frac 1 2 {_x\astrut_x})].
        \end{equation}
        It is easy to check that
        \[
            \Omega_x \inner_x(\exp(_x\astrut_y) \cup_x \frac 1 2 {_x\astrut_x})
            = \frac 1 2 {_y\astrut_y} \inner_y (\Omega_y \cup_y \exp(_x\astrut_y)).
        \]
        Furthermore, one has
        \[
            \Omega_x \inner_x \exp(_x\astrut_y) = \Omega_y \cup_y \exp(_x\astrut_y),
        \]
        so~\eqref{eq:cor} becomes
        \begin{equation}
            \label{eq:final}
            \begin{split}
                &\Omega_y \cup_y \tr_x [\exp(_x\astrut_y) \juxta_x (\Omega_x \inner_x \frac 1 2
                {_x\astrut_x})]\\
                =&  \tr_x [(\Omega_x \inner_x \exp(_x\astrut_y)) \juxta_x (\Omega_x \inner_x \frac 1 2 {_x\astrut_x})] \\
                =& \tr_x [\frac 1 2 {_y\astrut_y} \inner_y (\Omega_y \cup_y \exp(_x\astrut_y))] \\
                =&  \frac 1 2 {_y\astrut_y} \inner_y (\Omega_y \cup_y \tr_x [\exp(_x\astrut_y)]).
            \end{split}
        \end{equation}

        Applying the Rozansky--Witten weight system given by $(X, \sigma, \sheaf F)$ to~\eqref{eq:final} then yields
        \[
            \begin{split}
                \sah \wedge \tr_{\sheaf F} [\exp(\alpha)(\sheaf F) \circ (\sah \inner v) \circ \exp(\alpha)(\sheaf
                F)]
                & = v \inner (\sah \wedge \tr_{\sheaf F} [\exp(\alpha)(\sheaf F)]) \\
                & = v \inner (\sah \wedge \ch(\sheaf F)),
            \end{split}
        \]

        from which the claim of the Lemma easily follows.
    \end{proof}

    \begin{lemma}
        \label{lem:wheeling2}
        The following relation holds in the graph homology space $\mathcal A(\orcir_x \, \astar_y)$:
        \[
            \tr_x[(\Omega_x \inner_x X) \juxta_x (\Omega_x \inner_x Y)] = \tr_x [\Omega_x \inner_x (C \cup_x D)]
        \]
        for all $C \in \mathcal A(\astar_x \, \astar_y)$ and $D \in \sheaf A(\astar _x)$.
    \end{lemma}

    \begin{proof}
        Let $H_{z, x} \in \mathcal A(\astar_z \orcir_x)$ be the Kontsevich integral of a bead $x$ on a wire $z$
        (see~\cite{thurston}).
        It has the property
        \[
            \Delta^x_{x_1, x_2} H_{z, x} = H_{z, x_1} \juxta_z H_{z, x_2}
            \in \mathcal A(\astar_z \, \orcir_{x_1} \,
            \orcir_{x_2}).
        \]

        Define an operator
        \[
            \Phi\colon \mathcal A(\astar_x \, \orcir_y) \to \mathcal A(\astar_z \, \orcir_y),
            \ X \mapsto \langle H_{z, x}, X\rangle_x := \langle \tilde H_{z, x}, X\rangle_x,
        \]
        where $\tilde H_{z, x} \in \mathcal A(\astar_x \, \astar_z)$ is any lift of $H_{z, x}$ under the map
        $\tr_x \circ \pi_x\colon \mathcal A(\astar_z \, \astar_x) \to \mathcal A(\astar_z \, \orcir_x)$. (That
        the operator is independent of the chosen lift follows from the same considerations
        as in the proof of~\cite[Lemma 5.4]{thurston}.)
        By \cite[Thm.\ 4]{thurston}, we moreover have
        \[
            H_{z, x} = \tr_x [\exp(_x\astrut_z \cup_x \Omega_x)].
        \]
        It follows that
        \[
            \Phi(C) = \langle (\exp(_x\astrut_z) \cup_x \Omega_x, C\rangle_x = \Delta_z^x (\Omega_x \inner_x C).
        \]
        Thus one has
        \[
            \begin{split}
                &\tr_y\tr_x [(\Omega_x \inner C) \juxta_x (\Omega_x \inner_x
                D)]\\
                 =& \tr_x [(\Omega_x \inner \tr_y(C)) \juxta_x (\Omega_x \inner_x D)] \\
                =&  \tr_x [\langle H_{x, x_1} \juxta_x H_{x, x_2},
                \Delta^x_{x_1} \tr_y(C) \otimes \Delta^x_{x_2}(D)\rangle_{x_1, x_2}] \\
                =&  \tr_x [\langle\Delta_{x_1, x_2}^z H_{x, z}, \Delta^x_{x_1}\tr_y(C) \otimes \Delta^x_{x_2}(D)\rangle_{x_1, x_2}] \\
                =&  \tr_x [\Delta^z_x \langle H_{z, x}, \tr_y (C) \cup_x D\rangle_x] \\
                =&  \tr_x [\Omega_x \inner_x (\tr_y(C) \cup_x D)] \\
                =&  \tr_y \tr_x [\Omega_x \inner_x (C \cup_x D)].
            \end{split}
        \]
        Finally recall from~\cite{thurston}
        that $\tr_y\colon\mathcal A(\orcir_x \, \astar_y) \to \mathcal A(\orcir_x \, \orcir_y)$
        is an isomorphism.
    \end{proof}

    \bibliographystyle{amsplain}
    \bibliography{my}

\providecommand{\bysame}{\leavevmode\hbox to3em{\hrulefill}\thinspace}
\providecommand{\MR}{\relax\ifhmode\unskip\space\fi MR }
\providecommand{\MRhref}[2]{%
  \href{http://www.ams.org/mathscinet-getitem?mr=#1}{#2}
}
\providecommand{\href}[2]{#2}
\begin{thebibliography}{10}

\bibitem{beauville:Progr}
A.~Beauville, \emph{{Some remarks on {K}\"ahler manifolds with {$c\sb{1}=0$}}},
  Classification of algebraic and analytic manifolds (Katata, 1982) Progr.
  Math.\ 39. 1--26, 1983.

\bibitem{beauville:kummer}
\bysame, \emph{{Vari\'et\'es k\"ahleriennes dont la premi\`ere classe de Chern
  est nulle}}, J. Diff. Geom. \textbf{18} (1983), 755--782.

\bibitem{bogomolov}
F.A. Bogomolov, \emph{{On the cohomology ring of a simple hyperk\"ahler
  manifold (on the results of Verbitsky)}}, Geom. Funct. Anal. \textbf{6}
  (1996), no.~4, 612--618.

\bibitem{buchweitz-flenner}
R.-O. Buchweitz and H.~Flenner, \emph{{The global decomposition theorem for
  Hochschild (co-)homology of singular spaces via the Atiyah--Chern
  character}}, Adv. Math. \textbf{217} (2008), no.~1, 243--281.

\bibitem{caldararu1}
A.~C{\u a}ld{\u a}raru, \emph{{The Mukai pairing, I: The Hochschild
  structure}}, {Preprint \tt arXiv:math.AG/0308079}.

\bibitem{caldararu:mukai2}
\bysame, \emph{{The Mukai pairing., II: The Hochschild--Kostant--Rosenberg
  isomorphism}}, Adv. Math. \textbf{194} (2005), no.~1, 34--66.

\bibitem{vdbergh}
M.~Van den Bergh and D.~Calaque, \emph{{Hochschild cohomology and Atiyah
  classes}}, {Preprint \tt arXiv:0708.2725}.

\bibitem{fujiki}
A.~Fujiki, \emph{{On the de Rham cohomology group of a compact K\"ahler
  symplectic manifold}}, {Algebraic geometry, Proc. Symp., Sendai/Jap. 1985,
  Adv. Stud. Pure Math. 10, 105--165.}, 1987.

\bibitem{ghj}
M.~Gross, D.~Huybrechts, and D.~Joyce, \emph{{Calabi--Yau manifolds and related
  geometries. Lectures at a summer school in Nordfjordeid, Norway, June 2001}},
  {Universitext. Berlin: Springer.}, 2003.

\bibitem{hitchin-sawon}
N.~Hitchin and J.~Sawon, \emph{{Curvature and characteristic numbers of
  hyper-K\"ahler mani\-folds}}, Duke Math. J. \textbf{106} (2001), no.~3,
  599--615.

\bibitem{huybrechts:fm}
D.~Huybrechts, \emph{{Fourier-Mukai transforms in algebraic geometry}}, {Oxford
  Mathematical Monographs.}, 2006.

\bibitem{huybrechts:projective}
D.~Huybrechts and R.~Thomas, \emph{{$\mathbb P$-objects and autoequivalences of
  derived categories}}, Math. Res. Lett. \textbf{13} (2006), no.~1, 87--98.

\bibitem{kapranov}
M.~Kapranov, \emph{{Rozansky-Witten invariants via Atiyah classes}}, Compositio
  Math. \textbf{115} (1999), 71--113.

\bibitem{kontsevich}
M.~Kontsevich, \emph{{Deformation quantization of Poisson manifolds}}, Lett.
  Math. Phys. \textbf{66} (2003), no.~3, 157--216.

\bibitem{namikawa}
Y.~Namikawa, \emph{{Counter-example to global Torelli problem for irreducible
  symplectic manifolds}}, Math. Ann. \textbf{324} (2002), no.~4, 841--845.

\bibitem{nieper:book}
M.~Nieper-Wi{\ss}kirchen, \emph{Chern numbers and {R}ozansky--{W}itten
  invariants of compact hyper-{K}\"ahler manifolds}, World Scientific
  Publishing Co. Inc., 2004.

\bibitem{orlov}
D.O. Orlov, \emph{{Derived categories of coherent sheaves and equivalences
  between them}}, Russ. Math. Surv. \textbf{58} (2003), no.~3, 511--591.

\bibitem{ploog:thesis}
D.~Ploog, \emph{{Equivariant autoequivalences for finite group actions}}, Adv.
  Math. \textbf{216} (2007), no.~1, 62--74.

\bibitem{ramadoss}
A.~Ramadoss, \emph{{The relative Riemann--Roch theorem from Hochschild
  homology}}, {Preprint \tt arXiv:math/0603127}.

\bibitem{rowi}
J.~Roberts and S.~Willerton, \emph{{On the Rozansky-Witten weight systems}},
  {Preprint \tt arXiv:math.AG/0602653}.

\bibitem{thurston}
D.~Thurston, \emph{{A Diagrammatic analogue of the Duflo isomorphism}}, Ph.D.
  thesis, {University of California at Berkeley}, 2000.

\bibitem{ueno}
K.~Ueno, \emph{{Classification theory of algebraic varieties and compact
  complex spaces}}, {Lecture Notes in Mathematics. 439. Springer-Verlag.},
  1975.

\end{thebibliography}

\end{document}